\documentclass[12pt,a4paper,twoside]{amsart}
\usepackage[utf8]{inputenc}
\usepackage[normalem]{ulem}
\usepackage[english]{babel}
\usepackage{amsfonts, amsmath, amssymb, amsthm, amscd}
\usepackage{upgreek}
\usepackage{anysize}
\marginsize{2cm}{2cm}{1.5cm}{1.5cm}
\usepackage{enumerate}
\usepackage{hyperref}
\hypersetup{
colorlinks   = true, 
urlcolor     = blue, 
linkcolor    = blue, 
citecolor    = red 
}
 
\newcommand{\Href}[2]{\hyperref[#2]{#1~\ref{#2}}}

\numberwithin{equation}{section}
\newtheorem{thm}{Theorem}[section]
\newtheorem{prop}[thm]{Proposition}
\newtheorem{lem}[thm]{Lemma}

\newtheorem{cor}[thm]{Corollary}

\theoremstyle{definition}
\newtheorem{dfn}[thm]{Definition}
\newtheorem{remark}[thm]{Remark}
\newtheorem*{basicassumptions*}{Basic Assumptions}

\newtheorem*{cosmeticassumptions*}{Auxiliary Assumptions}

\DeclareMathOperator{\dom}{dom}
\DeclareMathOperator{\supp}{supp}
\DeclareMathOperator{\conv}{conv}
\DeclareMathOperator{\logconc}{\operatorname{log-env}}

\newcommand{\Red}{\Re^d}
\newcommand{\Ze}{\mathbb Z}
\newcommand{\ball}[1]{\mathbf{B}^{#1}}
\newcommand{\norm}[1]{\left\|#1\right\|}
\newcommand{\enorm}[1]{\left|#1\right|}
\newcommand{\bd}[1]{\mathrm{bd}\left(#1\right)}
\newcommand{\cl}[1]{\mathrm{cl}\left(#1\right)}

\newcommand{\interior}{\operatorname{int}}
\newcommand{\iprod}[2]{\left\langle#1,#2\right\rangle}
\newcommand{\tr}[1]{\operatorname{tr}#1}

\newcommand{\nfcone}[2]{N\!\left(#1, #2\right)}
\newcommand{\poscone}[1]{\operatorname{Pos}\! \parenth{#1}}

\def\N{{\mathbb N}}
\def\R{{\mathbb R}}

\renewcommand{\Re}{\mathbb{R}}

\newcommand{\Redp}{\Re^{d+1}}

\newcommand{\st}{\colon}

\def\phi{\varphi}
\def\epsilon{\varepsilon}
\newcommand{\abs}[1]{\left\lvert#1\right\rvert}
\newcommand{\transpose}[1]{{#1}^{T}}%

\newcommand{\di}{\,\mathrm{d}}

\newcommand{\upthing}[1]{\overline{#1}}%
\newcommand{\lifting}[1]{{{\operatorname{Lift}#1}}}%
\newcommand{\funpos}[1]{\mathcal{E}\! \left[ #1\right]}
\newcommand{\funppos}[1]{\mathcal{E}^{+}\! \left[ #1\right]}
\newcommand{\funposfc}[1]{\mathcal{E}_{f.c.}\! \left[ #1\right]}
\newcommand{\funpposfc}[1]{\mathcal{E}_{f.c.}^{+}\! \left[ #1\right]}

\newcommand{\id}{\mathrm{Id}}
\newcommand{\epi}{\mathrm{epi}}
\newcommand{\essgraph}{\mathrm{ess}\ \mathrm{graph}\ }
\newcommand{\hausmetric}[2]{\delta_{H}\parenth{#1, #2}}%
\def\legendre{\mathcal{L}}%
\newcommand{\loglego}[1]{{#1}^\polar}
\def\polar{\circ}
\newcommand{\slogleg}[2][s]{
\legendre_{ #1} #2}%

\def\epi{\operatorname{epi}}

\providecommand{\parenth}[1]{\left(#1\right)}
\newcommand{\littleo}[1]{o \! \parenth{#1}}%
\newcommand{\MM}{\mathcal{M}}
\newcommand{\MMbar}{\mathcal{W}}

\newcommand{\sid}[1][d]{{\id}_{#1} \oplus s}%
\newcommand{\projond}{P_d}
\newcommand{\derivativeatzero}{\left.\frac{\di}{\di t}\right|_{t=0+}}%

\newcommand{\contactpoint}[2]{{C}_{point}\!\parenth{#1, #2}}%
\newcommand{\contactopjohn}[2]{C_J \! \parenth{#1, #2}}%
\newcommand{\contactoplowner}[2]{C_L \! \parenth{#1, #2}}%

\newcommand{\contactsetnr}[2]{\mathcal{C} \! \parenth{#1, #2}}%
\newcommand{\contactset}[2]{\mathcal{C}_{red} \! \parenth{#1, #2}}%
\newcommand{\contactopsetjohn}[2]{\mathcal{A}_J \! \parenth{#1, #2}}%
\newcommand{\contactopsetlowner}[2]{\mathcal{A}_L \! \parenth{#1, #2}}%

\newcommand{\locstar}{star-like subset of }
\newcommand{\wlocstar}{star-like set with respect to~}
\newcommand{\gbase}{g_b}

\title[Functional John and L{\"o}wner conditions]{Functional John and L{\"o}wner conditions for pairs of log-concave functions}
\author[G. Ivanov and M. Nasz{\'o}di]{Grigory Ivanov\address{G.I.: 
Institute of Science and Technology Austria (ISTA), 
Klosterneuburg, Austria}
\email{grimivanov@gmail.com}
\and
M{\'a}rton Nasz{\'o}di\address{M.N.:
Alfr\'ed R\'enyi Inst. of Mathematics and
Dept. of Geometry, Lor\'and E\"otv\"os University, Budapest}
\email{marton.naszodi@math.elte.hu}
}

\thanks{
M.N. was supported by the J\'anos Bolyai Scholarship of the Hungarian Academy of Sciences as well as the National Research, Development and Innovation Fund (NRDI) grants K119670, and K131529, and the \'UNKP-22-5 New National Excellence Program of the Ministry for Innovation and Technology from the source of the NRDI.
}


\begin{document}

\begin{abstract}
John's fundamental theorem characterizing the largest volume ellipsoid contained in a convex body $K$ in $\mathbb{R}^d$ has seen several generalizations and extensions. One direction, initiated by V. Milman is to replace ellipsoids by positions (affine images) of another body $L$. Another, more recent direction is to consider logarithmically concave functions on $\mathbb{R}^d$ instead of convex bodies: we designate some special, radially symmetric log-concave function $g$ as the analogue of the Euclidean ball, and want to find its largest integral position under the constraint that it is pointwise below some given log-concave function $f$.

We follow both directions simultaneously: we consider the functional question, and allow essentially any meaningful function to play the role of $g$ above. Our general theorems jointly extend known results in both directions.

The dual problem in the setting of convex bodies asks for the smallest volume ellipsoid, called \emph{L{\"o}wner's ellipsoid}, containing $K$. We consider the analogous problem for functions: we characterize the solutions of the optimization problem of finding a smallest integral position of some log-concave function $g$ under the constraint that it is pointwise above $f$. It turns out that in the functional setting, the relationship between the John and the L{\"o}wner problems is more intricate than it is in the setting of convex bodies.
\end{abstract}
\maketitle

\section{Introduction}

The largest volume ellipsoid contained in a convex body in $\Red$ and, in particular, John's result \cite{john} 
characterizing it, plays a fundamental role in convexity. The latter states that the 
origin-centered Euclidean unit ball is the largest volume ellipsoid contained 
in the convex body $K$ if and only if it is contained in $K$ and the contact 
points (that is, the intersection points of the unit sphere and the boundary of 
$K$) satisfy a certain algebraic condition.

As a natural generalization, one may fix two convex bodies $K$ and $L$ and 
solve the optimization problem of finding a largest volume affine image of $K$ 
contained in $L$. In this setting, one expects a John-type condition in terms of 
contact pairs defined as follows.
If $K\subseteq L\subset\Red$ are convex bodies, then a pair 
$(u,v)\in\Red\times\Red$ is called a \emph{contact pair}, if
$u$ belongs to the intersection of the boundaries of $K$ and $L,$
$v$ belongs to the intersection of the boundaries of the polar sets  $\loglego{K}$ and $\loglego{L},$ and $\iprod{u}{v}=1$. In other words, $v$ is 
an outer normal vector of a common support hyperplane of $K$ and $L$ at a common 
boundary point $u$, with a proper normalization.

V. Milman achieved the first results giving a condition of optimality in the 
above problem (unpublished, see Theorem~14.5 in \cite{TJ89}) followed by Lewis 
\cite{L79}, which were strengthened and extended by Giannopoulos, Perissinaki, 
and Tsolomitis \cite{GPT01} and then by Bastero and Romance \cite{BR02}. 
Finally,
Gordon, Litvak, Meyer, and Pajor \cite[Theorem~3.1]{GLMP04} proved the following.

\begin{thm}[Gordon, Litvak, Meyer, Pajor \cite{GLMP04}]\label{thm:GLMP}
	Let $K$ be a compact set containing the origin in the interior of its convex hull and $L$ be a convex body in $\Red$ with $K\subseteq L$ such that no 
affine image of $\conv{K}$ contained in $L$ is of larger volume than $\conv{K}$. Then there are contact pairs 
$\{(u_i,v_i)\st i=1,\ldots m\}$ of $K$ and $L$ with $m\leq d^2+d$ such that we have
	\begin{equation*}
	\sum_{i=1}^{m} c_i u_i\otimes v_i=\id_d, \text{ and }\;
	\sum_{i=1}^{m} c_i v_i=0,
	\end{equation*}
 where $\id_{d}$ denotes the identity operator on $\Red$, and $u\otimes v$ denotes 
the linear operator $x\mapsto \iprod{v}{x}u$ for every $x\in\Red$.
\end{thm}

As a corollary, it is shown in \cite{GLMP04} that under the assumptions of 
\Href{Theorem}{thm:GLMP}, we have $L-z\subseteq -d(K-z)$ for an appropriately 
chosen point $z\in K$. The case when $-d$ cannot be 
replaced by a magnification factor of smaller absolute value was studied in 
\cite{Pal92, JN11}, see Gr\"unbaum's survey \cite{Gr63} for more on this 
question.

Our present goal is to extend \Href{Theorem}{thm:GLMP} to the setting of 
log-concave functions. As a natural generalization of the notion of 
affine images of convex bodies, we define the \emph{positions} of a function $g$ 
on $\Red$ as
\[\funpos{g}=\{\alpha g(Ax+a)\st A\in\Re^{d\times d} \text{ non-singular}, \alpha>0, a\in\Red\}.\]
We will say that a function $f_1$ on $\Red$ is \emph{below} another function $f_2$ on $\Red$ (or that $f_2$ is \emph{above} $f_1$)
and denote it as $f_1 \leq f_2,$ if $f_1$ is pointwise less than or equal to $f_2,$
that is, $f_1(x) \leq f_2(x)$ for all $x \in \Red.$

Fixing $s > 0$ and two functions $f, g  \colon \Red \to [0, \infty)$, we 
formulate the following optimization problem. 

\medskip \noindent
\textbf{The John $s$-problem:} Find
\begin{equation}\label{eq:john_problem_intro}
\max\limits_{h \in \funpos{g} }
	\int_{\Red} h^s 
	\quad \text{subject to} \quad
	h \leq f.
\end{equation}

John's theorem on largest volume ellipsoids was extended by 
Alonso-Guti{\'e}rrez, Gonzales Merino, Jim{\'e}nez and Villa 
\cite{alonso2018john} to the setting of \emph{logarithmically concave} 
(or in short, \emph{log-concave}) functions, that is, those $\Red\to[0,\infty)$ 
functions whose logarithm is a concave function. They consider the John $1$-problem with  $g$ being the indicator function of the Euclidean unit ball $\ball{d}.$
A more general treatment was given in \cite{ivanov2022functional}, where the authors 
consider the John $s$-problem for some $ s > 0$ with $g$ being the ``height'' function of the upper hemisphere of the Euclidean ball $\ball{d+1},$ that is,  $g(x) = \sqrt{1 - \enorm{x}^2}$ for  $\enorm{x} \leq 1$ and $g(x)= 0$ otherwise. The authors obtained
 the first necessary and sufficient condition on maximizers in this problem analogous to the original John condition.  
 
 What conditions should $f$ and $g$ satisfy in order for the John $s$-problem to be meaningful? Following the path of analogy with \Href{Theorem}{thm:GLMP} seems easy at first. Instead of closed sets, we will work with upper semi-continuous functions, instead of volume, we will work with the integral, or the integral of the $s$ power of the function. In \Href{Theorem}{thm:GLMP}, $K$ and $L$ are compact, should we assume that the \emph{support} of $f$ and $g$ defined as 
\[
\supp f = \{x \in \Red \st f(x) > 0 \}
\]
is bounded? That would be too restrictive, as it would disqualify the Gaussian density as $f$. On the other hand, clearly, the class of those functions $g$, for which the family $\{h\in\funpos{g}\st h\leq f\}$ is not empty \emph{for any} log-concave function $f$ with positive integral, is the class of functions with bounded support. Thus, $g$ being of bounded support is a natural assumption.

We will call an upper semi-continuous  function of finite and positive integral a \emph{proper} function.
  
\begin{basicassumptions*}
\label{assumptions:basic}
 We say that a function $g  \colon \Red \to [0, + \infty)$ satisfies our
 \emph{Basic Assumptions}, if it has the following properties:
\begin{itemize}
\item $g$ is a proper log-concave function, and
\item $\supp g$ is bounded, and
\item the origin is in the interior of $\supp g$.
\end{itemize} 
\end{basicassumptions*}
\begin{cosmeticassumptions*}
\label{assumptions:intro} 
 We say that a function $g  \colon \Red \to \Re$ satisfies our
 \emph{Auxiliary  Assumptions} if it has the following properties:
\begin{itemize}
\item $g$ satisfies our Basic Assumptions, and
\item $g$ attains its maximum at the origin, and
\item $\ln g$ is differentiable on $\supp g$.
\end{itemize} 
\end{cosmeticassumptions*}

\begin{dfn}\label{defn:contacts_point}
For two functions $f,g\colon\Red\to\Re$, we call the set
		\[
		\contactpoint{f}{g} = 
		\big\{ u \in \cl{\supp f} \cap \cl{\supp g} \st  f(u)  = g(u)\big\}
		 \]
their \emph{set of contact points}.
\end{dfn}


We are ready to state our first main result. 

\begin{thm}[John's condition -- no zeros]\label{thm:john_intro}
Fix $s>0$. 
Let $f \colon\Red\to (0,+\infty)$ be a proper log-concave function taking only positive values. Let $g = e^{-\psi} \colon \Red \to [0, +\infty)$ be a function satisfying our Auxiliary Assumptions (see page~\pageref{assumptions:intro}) such that $g \leq f$.
Assume that $h=g$ is a maximizer in John $s$-problem 
\eqref{eq:john_problem_intro}.
Then there are contact
points $u_1, \dots, u_m \in \contactpoint{f}{g}$ 
and positive weights $c_1,\ldots,c_m$ such that
\begin{equation}\label{eq:functional_glmp_intro}
	\sum_{i=1}^{m} c_i \frac{{u}_i \otimes \nabla \psi(u_i)}{1 + \iprod{\nabla \psi(u_i)}{u_i}} = 
	\id_{d}, \quad 
	\sum_{i=1}^{m}  \frac{c_i}{1 + \iprod{\nabla \psi(u_i)}{u_i}} =  s
		\quad\text{and}\quad
		\sum_{i=1}^{m} c_i \frac{\nabla \psi(u_i)}{1 + \iprod{\nabla \psi(u_i)}{u_i}}=0.
\end{equation}
Moreover, if $g$ is \emph{radially symmetric}, then condition \eqref{eq:functional_glmp_intro} is also sufficient. That is,  if $g(x)=g_0(|x|)$ for some function $g_0:[0,\infty)\rightarrow[0,\infty)$, and there are contact
points $u_1, \dots, u_m \in \contactpoint{f}{g}$ 
and positive weights $c_1,\ldots,c_m$ satisfying \eqref{eq:functional_glmp_intro}, then 
$g$ is the unique maximizer in John $s$-problem 
\eqref{eq:john_problem_intro}.
\end{thm}

\Href{Theorem}{thm:john_intro} will be a corollary to our much more general 
\Href{Theorem}{thm:john_condition_general}. 

Let us elaborate on the conditions on the functions. 
First, the differentiability of $\psi$ is assumed for simplicity, in the general setting it will not be necessary as we will consider subgradients of $\psi$, and the 
Fr\'echet normal cones (see \Href{Definition}{def:nconus}) of the \emph{lifting} of $f$ and $g$, defined as 
\[
\lifting{f} =  \left\{ (x, y) \in \Redp \st x \in \cl{\supp f}, |y|\leq f(x) \right\}\subset\Redp.
\]

Second, as in the case of convex sets, the origin must be chosen in a certain way. 
The assumption that $g$ attains its maximum at the origin is artificial and is imposed for simplicity; as we will see, it implies that all the denominators in the equations in
\eqref{eq:functional_glmp_intro} are positive, which is the real key condition for our theorem to hold. In fact, we will show that any point from the interior of the support of $g$ sufficiently close to the maximum can be chosen as the origin.

Third, analogously to \Href{Theorem}{thm:GLMP}, where $K$ need not be convex, our $g$ need not be log-concave. We will have an analogue of the convex hull as well, the log-concave envelope, see \Href{Definition}{def:logenv}.

Finally, the assumption that $f$ takes only positive values is the trickiest one!  The issue is that there might be ``irregular'' contact points 
$u \in \cl{\supp f} \cap \cl{\supp g}$ with $f(u)=g(u)=0$ that require special attention, as we will see in \Href{Section}{sec:boundedcontactpairs}.
\Href{Theorem}{thm:john_intro} side steps this problem by its assumption that $f$ is nowhere zero.

To obtain a condition for optimality in John $s$-problem \eqref{eq:john_problem_intro} similar to the one in \Href{Theorem}{thm:GLMP}, we define contact pairs for functions through contacts of their liftings. Since liftings of log-concave functions are not 
convex in general, we need to take extra care defining normal vectors, which we 
will do in \Href{Section}{sec:notation}.
\begin{dfn}\label{defn:contactset_intro}
For two functions $f,g\colon\Red\to\Re$, their \emph{set of contact pairs} is defined as		\[
		\contactsetnr{f}{g} = 
		\big\{(\upthing{u},\upthing{v})\in\Redp\times\Redp \st 
		 \upthing{u}=(u, f(u)), u \in \cl{\supp {f}}\cap\cl{\supp {g}},\,  f(u)  = g(u),
		 \]
		 \[
		 \upthing{v} \in \nfcone{\lifting f}{\upthing{u}}\cap \nfcone{\lifting g}{\upthing{u}}, 
		 \, \iprod{\upthing{v}}{\upthing{u}}=1\big\},
		\]
where $\nfcone{A}{\upthing{u}}$ denotes the  Fr\'echet normal cone of the set 
$A\subset\Redp$ at a point $\upthing{u} \in \Redp,$ see 
\Href{Definition}{def:nconus}.
\end{dfn}

In the following theorem, no additional assumptions are imposed on $f$ except for being proper and log-concave. On the other hand, we require $g$ to be \emph{$q$-concave}, that is, $g^q$ is concave on its support.   
\begin{thm}[John's condition -- $q$-concave case]\label{thm:john_intro-concave}
Fix $s>0$. Let $f,g \colon\Red\to[0,+\infty)$ be two proper log-concave functions.
Let $g = e^{-\psi} \colon \Red \to [0, +\infty)$ be a function satisfying our Auxiliary Assumptions (see page~\pageref{assumptions:intro})  and such that $g \leq f.$ Additionally, let $g$ be $q$-concave for some $q > 0$.
Assume that $h=g$ is a maximizer in  John $s$-problem 
\eqref{eq:john_problem_intro}.
Then there are contact pairs 
$(\upthing{u}_1,\upthing{v}_1)$, $\dots$,  
$(\upthing{u}_m,\upthing{v}_m)$ $\in \contactsetnr{f}{g}$
and positive weights $c_1,\ldots,c_m$ such that
\begin{equation}\label{eq:functional_glmp_concave_intro}
	\sum_{i=1}^{m} c_i {u}_i \otimes {v}_i = 
	\id_{d}, \quad 
	\sum_{i=1}^{m} c_i f(u_i)\nu_i =  s
		\quad\text{and}\quad
		\sum_{i=1}^{m} c_i v_i=0,
\end{equation}
where $\upthing{u}_i=(u_i, f(u_i))$ and $\upthing{v}_i=(v_i,\nu_i)$.
Moreover, if $g$ is radially symmetric, then condition \eqref{eq:functional_glmp_concave_intro} is also sufficient. That is,  if $g$ is  radially symmetric, and  there are contact pairs 
$(\upthing{u}_1,\upthing{v}_1)$, $\dots$,  
$(\upthing{u}_m,\upthing{v}_m)$ $\in \contactsetnr{f}{g}$
and positive weights $c_1,\ldots,c_m$ satisfying \eqref{eq:functional_glmp_concave_intro}, then 
$g$ is the unique maximizer in John $s$-problem 
\eqref{eq:john_problem_intro}.
\end{thm}

A dual construction to the largest volume ellipsoid contained in a convex body
is the smallest volume ellipsoid containing a body.  It is usually called the L\"owner ellipsoid. Notably, the necessary and sufficient conditions for the Euclidean unit ball to be this minimal ellipsoid coincide with the conditions in John's characterization of the largest volume ellipsoid. For historical precision, we remark that it was this formulation, now attributed to L\"owner, that John considered in the first place in \cite{john}. In the setting of convex sets, there is hardly any difference between the two problems --  $K$ has the largest volume among all its affine images inside $L$ if and only if $L$ has the smallest volume among all its affine images containing $K$. So, \Href{Theorem}{thm:GLMP} provides us with a necessary condition in this case as well.   However, it is not the case in the functional setting!  Let us formulate a dual functional problem and  explain the issue. 

\medskip \noindent
\textbf{The L\"owner $s$-problem:} Find
\begin{equation}
\label{eq:lowner_problem_intro}
\min\limits_{h \in \funpos{g} }
	\int_{\Red} h^s 
	\quad \text{subject to} \quad
	f \leq h.
\end{equation}

As in the case of the John $s$-problem,  the set of $h \in  \funpos{g}$ satisfying $f \leq h$ may be empty, for example, if $g$ is of compact support and the support of $f$ is the whole space $\Red$. Unlike in the case of the John $s$-problem, characterizing the class of those functions $g$, for which the family $\{h\in\funpos{g}\st h\geq f\}$ is not empty \emph{for any} proper log-concave function $f$, is not straight forward. Clearly, the support of $g$ needs to be $\Red$, but this condition alone is not sufficient. In order to find this class, we consider the polars of $f$ and $g$.

The \emph{log-conjugate} (or \emph{polar}) of a function $f\colon\Red \to [0, +\infty)$ is 
defined by
\[
\loglego{f}(y) = \inf\limits_{x \in \supp{f}} \frac{e^{-\iprod{x}{y}}}{f(x)},
\]
and is known to be a log-concave function, see \Href{Section}{sec:notation} for details. 

Since $h\geq f$ if and only if $\loglego{h}\leq\loglego{f}$, it follows that the class of those log-concave functions $g$, for which the family $\{h\in\funpos{g}\st h\geq f\}$ is not empty \emph{for any} proper log-concave function $f$, is the class of log-concave functions that are polar to functions with bounded support. 
 
The L\"owner $s$-problem was investigated by Li, Sch\"utt and Werner \cite{LSW21} and by Ivanov and Tsiutsiurupa \cite{ivanov2021functional} for certain special choices of $g$. We note that no John type condition of optimality was obtained. Our second main result provides it.

\begin{thm}[L\"owner's condition -- no zeros]\label{thm:lowner_intro}
Fix $s>0$. 
Let $f \colon\Red\to (0,+\infty)$ be a proper log-concave function such that 
$\loglego{f}$ takes only positive values.  
Let $g \colon \Red \to [0, +\infty)$ be a proper log-concave function such that $f \leq g$ and $\loglego{g}$ satisfies our Auxiliary Assumptions (see page~\pageref{assumptions:intro}).
Set $\loglego{g} = e^{-\psi}$, and assume that $h=g$ is a minimizer in L\"owner $s$-problem 
\eqref{eq:lowner_problem_intro}.
Then there are contact
points $u_1, \dots, u_m \in \contactpoint{\loglego{f}}{\loglego{g}}$ 
and positive weights $c_1,\ldots,c_m$ such that
\begin{equation}\label{eq:functional_glmp_lowner_intro}
	\sum_{i=1}^{m} c_i \frac{{u}_i \otimes \nabla \psi(u_i)}{1 + \iprod{\nabla \psi(u_i)}{u_i}} = 
	\id_{d}, \quad 
	\sum_{i=1}^{m}  \frac{c_i}{1 + \iprod{\nabla \psi(u_i)}{u_i}} =  s
		\quad\text{and}\quad
		\sum_{i=1}^{m} c_i  \frac{\loglego{g}(u_i) \cdot \nabla \psi(u_i)}{1 + \iprod{\nabla \psi(u_i)}{u_i}}=0.
\end{equation}
Moreover, if $g$ is radially symmetric, then condition 
\eqref{eq:functional_glmp_lowner_intro} is also sufficient. That is,  if $g$ is  radially symmetric, and  there are contact
points $u_1, \dots, u_m \in \contactpoint{\loglego{f}}{\loglego{g}}$ 
and positive weights $c_1,\ldots,c_m$ satisfying \eqref{eq:functional_glmp_lowner_intro}, then $g$ is a minimizer in L\"owner $s$-problem \eqref{eq:lowner_problem_intro}.
\end{thm}

\begin{thm}[L\"owner's condition -- $q$-concave case]\label{thm:lowner_intro-concave}
Fix $s>0$. 
Let $f \colon\Red\to (0,+\infty)$ be a proper log-concave function.  
Let $g \colon \Red \to [0, +\infty)$ be a proper log-concave function such that $f \leq g$ and $\loglego{g}$ satisfies our Auxiliary Assumptions (see page~\pageref{assumptions:intro}).
Additionally, assume that $\loglego{g}$ is $q$-concave with some $q>0$. 
Assume also that $h=g$ is a minimizer in  L\"owner $s$-problem \eqref{eq:lowner_problem_intro}. 
Then there are contact pairs 
$(\upthing{u}_1,\upthing{v}_1), \dots,  
(\upthing{u}_m,\upthing{v}_m) \in \contactsetnr{\loglego{g}}{\loglego{f}}$
and positive weights $c_1,\ldots,c_m$ such that
\begin{equation}\label{eq:functional_glmp-lowner-conc_intro}
	\sum_{i=1}^{m} c_i {v}_i \otimes {u}_i = \id_{d}, \quad 
	\sum_{i=1}^{m} c_i \loglego{g}(u_i) \cdot \nu_i =  s
		\quad\text{and}\quad
		\sum_{i=1}^{m} c_i\loglego{g}(u_i) \cdot \nu_i u_i=0,
\end{equation} 
where $\upthing{u}_i=(u_i, \loglego{g}(u_i))$ and $\upthing{v}_i=(v_i,\nu_i)$.
Moreover, if $g$ is radially symmetric, then condition \eqref{eq:functional_glmp-lowner-conc_intro} is also sufficient. That is, if $g$ is  radially symmetric, and there are contact pairs 
$(\upthing{u}_1,\upthing{v}_1)$, $\dots$,  
$(\upthing{u}_m,\upthing{v}_m)$ $\in \contactsetnr{f}{g}$
and positive weights $c_1,\ldots,c_m$ satisfying \eqref{eq:functional_glmp-lowner-conc_intro}, then 
$g$ is a  maximizer in L\"owner $s$-problem 
\eqref{eq:lowner_problem_intro}.
\end{thm}

\Href{Theorems}{thm:lowner_intro} and \ref{thm:lowner_intro-concave} will be corollaries to our more general \Href{Theorem}{thm:lowner_condition_general}.

\begin{remark}[Duality and duality]
Observe that even though \Href{Theorem}{thm:lowner_intro} is phrased in 
terms of $\loglego{f}$ and $\loglego{g}$, it is not the same as 
\Href{Theorem}{thm:john_intro} for $\loglego{f}$ and 
$\loglego{g}$, even in the case where $s=1$. The reason is that we need to maximize/minimize a different 
functional -- the integral of the polar of the function instead of the integral 
of the function itself. In other words, the solution to L\"owner's problem is 
not the dual of the solution to John's problem. 
Moreover, comparing \eqref{eq:functional_glmp_lowner_intro} and \eqref{eq:functional_glmp_intro}, we see that the conditions are different. See more on this in \Href{Section}{sec:johnlownerequivalence}.

This is a major difference between our results and \Href{Theorem}{thm:GLMP}, 
since the latter has a self-dual form, \cite[Theorem~3.8]{GLMP04}. 
\end{remark}

\subsection{Structure of the paper}

In \Href{Section}{sec:notation}, we recall the basics of the theory of 
log-concave functions and polarity on functions.
Then, in \Href{Section}{sec:normalcones}, we discuss properties of normal cones 
of liftings of log-concave functions. These are rather technical facts, we 
suggest skipping the proofs on a first reading.
We state and prove our first main result, \Href{Theorem}{thm:john_condition_general}, the condition of 
optimality in John's problem in \Href{Section}{sec:john}, and 
our second main result, \Href{Theorem}{thm:lowner_condition_general}, the condition of optimality in the L\"owner's problem in \Href{Section}{sec:lowner}.

In \Href{Section}{sec:existence_uniqueness}, we show that the optima generally exist in both the John and the L\"owner problem, and discuss when uniqueness holds -- and when it does not.

\Href{Section}{sec:normalcone_subdifferential} describes the normal cone of the lifting of a log-concave function $e^{-\psi}$ in terms of the subddiferential of $\psi$. Then, in \Href{Section}{sec:boundedcontactpairs}, more readily applicable conditions on $g$ are shown that guarantee that the very technical conditions of \Href{Theorems}{thm:john_condition_general} and \ref{thm:lowner_condition_general} on $f$ and $g$ hold for essentially all meaningful choice of $f$.

In \Href{Sections}{sec:radially_symmetric} and \ref{sec:qconcave}, we present the preliminaries needed to prove the results of the Introduction on radially symmetric and $q$-concave functions.

\Href{Section}{sec:introproofs} contains the proofs of the results of the Introduction by combining the results of Sections \ref{sec:normalcone_subdifferential} through \ref{sec:qconcave} to show how our two main, general results \Href{Theorems}{thm:john_condition_general} and \ref{thm:lowner_condition_general} apply.

In \Href{Section}{sec:corollaries_and_discussion}, we note that \Href{Theorem}{thm:GLMP} follows from our results, and study the relationship of the John and L\"owner problems. Finally, we discuss what changes need to be made if affine positions of functions are replaced by linear positions, that is, when translations in $\Red$ are not allowed in the optimization problems.

\section{Basic notions}\label{sec:notation}

We use $\interior K$, $\bd K$, $\cl{K},$ and $\conv K$ to denote respectively the interior, boundary, closure and convex hull of a set $K$ in some Euclidean space, mostly $\Red$ or $\Redp$. 
We denote the Euclidean unit ball by $\ball{d}=\{x\in\Red\st |x|\leq 1\}$.
We will think of $\Red$ as the linear subspace of $\Redp$ spanned by the first $d$ elements of the standard basis. We denote the orthogonal projection from $\Redp$ to $\Red$ by $\projond$.
We use $e_{d+1}$ to denote the last vector of the standard basis of $\Redp.$

\subsection{Functions}
Let $f\colon\Red\to[0,\infty)$ be a function.
For $\alpha\in\Re$, we denote its $\alpha$ superlevel set by
\[
 [f\geq \alpha]=\{x\in\Red\st f(x)\geq\alpha\},
\]
and we use similar notations for level sets and sublevel sets of functions.
We denote the \emph{support} of $f$ by 
\[
\supp f = \{x \in \Red \st f(x) > 0 \}.
\]
The \emph{essential graph} and the \emph{lifting} of $f$ are the sets
\[
\essgraph f =  \left\{ (x, f(x)) \st x \in \cl{\supp f} \right\},\text{ and }
\]
\[
\lifting{f} =  \left\{ (x, y) \st x \in \cl{\supp f}, |y|\leq f(x) \right\}
\]
in $\Redp$.  

We call an upper semi-continuous function of finite and positive integral a \emph{proper} function. 
Note that for a proper log-concave function $f$, we have that $\lifting{f}$ is compact if and only if, $\cl{\supp f}$ is compact, which is equivalent to $f$ having bounded support.
A special class of log-concave functions is \emph{$q$-concave functions}\label{page:qconcave} for some $q>0$, which is those $f$ for which $f^q$ is  concave on its support. It is an exercise to show that if $f$ is $q$-concave and $0< r \leq q$, then $f$ is $r$-concave as well. 

The \emph{effective domain} of a convex function $\psi  \colon \Red \to \R \cup \{+ \infty\}$ is the set 
\[
\dom \psi = \left\{x \st  \psi(x) < + \infty\right\},
\]
the \emph{epigraph} of $\psi$ is
\[
\epi \psi = \left\{(x, \xi) \st x\in \dom \psi, \; \xi \in \R, \;    \xi \geq \psi(x) \right\}.
\]
Note that if $f=e^{-\psi}$ is a proper log-concave function, then $\psi$ is a lower semi-continuous convex function whose epigraph is closed, and
\begin{equation}\label{eq:suppequalsdom}
 \supp f =\dom \psi 
\end{equation}
are convex sets in $\Red$ with non-empty interior.


\begin{dfn}\label{def:logenv}
For a function $g \colon \Red  \to [0, + \infty)$, its \emph{log-concave envelope}, $\logconc{g}$, is the minimal upper semi-continuous log-concave function $h$ satisfying $g \leq h$. 
\end{dfn}
We note that for any function $g$, the log-concave envelope $\logconc{g}$ is well-defined.
The epigraph of $-\ln(\logconc{g})$ is the closure of the convex hull of the epigraph of $-\ln(g)$ in $\Redp$.


\subsection{Positions, Minkowski's determinant inequality}
We will work with positions of functions that are the analogues of affine images of convex bodies.
To this end, we define the vector space
    \[
        \MMbar=\{(A\oplus\alpha, a)  \st A\in\Re^{d\times d}, \alpha\in\Re \text{ and } a\in\Red\},
    \]
and its subsets
    \[
        \MM=\{(A\oplus\alpha, a)\in\MMbar  \st A \text{ is non-singular, and } \alpha>0\},
    \]
and
    \[
        \MM^+=\{(A\oplus\alpha, a)\in\MM  \st A \text{ is positive definite}\}.
    \]
We will refer to elements of $\MMbar$ as  \emph{extended contact operators}, and
we will say that $A \oplus \alpha$ is the \emph{operator part} and 
$a$ is the \emph{translation part} of $(A \oplus \alpha, a) \in \MMbar.$ 
    
We denote by
\[\funpos{g}=\{\alpha g(Ax+a)\st (A\oplus\alpha, a)\in\MM\}\]
the \emph{positions} of a function $g$ on $\Red$, and by
\[\funppos{g}=\{\alpha g(Ax+a)\st (A\oplus\alpha, a)\in\MM^+\}.\]
the \emph{positive positions} of $g$.

We recall the additive and the multiplicative forms of the
\emph{Minkowski determinant  inequality}. Let $A$ and $B$ be 
positive definite matrices of order $d$. Then, for any $\lambda \in (0,1),$
\begin{equation}
\label{eq:minkowski_det_ineq}
\left(\det\left( \lambda A + (1 - \lambda)B\right)\right)^{1/d} \geq
\lambda \left(\det A\right)^{1/d} + 
(1 -\lambda)\left(\det B\right)^{1/d},
\end{equation}
with equality if and only if $A = cB$ for some $c > 0;$ and 
\begin{equation}
\label{eq:minkowski_det_multipl_ineq}
\det\left( \lambda A + (1 - \lambda)B\right) \geq
\left(\det A\right)^{\lambda} \cdot \left(\det B\right)^{1 -\lambda},
\end{equation}
with equality if and only if $A = B$.

\subsection{Polarity for sets and functions}\label{sec:polarity}
Recall that the polar of a set 
 $K$ in $\Red$ is given by
 \[
 {K}^\polar = \{y \in \Red \st \iprod{x}{y} \leq 1
  \quad \text{for all } \ x \in K \}.
 \]
For any cone $C$ with apex at the origin of a linear space $L$, we call its polar set the \emph{polar cone}. It is easy to see that
\[
C^\circ = \left\{p \in L^{\ast} \st \iprod{p}{a} \leq 0 \quad \text{ for all } a \in C\right\}.
\]

The classical \emph{convex conjugate} transform (or \emph{Legendre transform})
$\legendre$ is defined for a function $\phi: \Red \to \R\cup \{+\infty\}$ by
\[
\slogleg[]{\phi}(y) = \sup\limits_{x \in \Red} \{\iprod{x}{y} - \phi(x)\}.
\]
This notion yields the following duality mapping on the set of log-concave functions, justified in~\cite{artstein2007characterization, artstein2008concept, artstein2009concept}.
The \emph{log-conjugate} (or \emph{polar}) of a log-concave function $f = e^{-\psi} : \Red \to [0, +\infty)$, is defined by
\[
\loglego{f}(y) = e^{- (\slogleg[] \psi)(y)} = 
\inf\limits_{x \in \supp{f}} \frac{e^{-\iprod{x}{y}}}{f(x)}.
\]

\subsection{The normal cone}\label{sec:normalcone}

For a set $S$ in a linear space, we define its \emph{positive cone} 
(or \emph{cone hull}) by
\[
\poscone{S} = \left\{\lambda a \st \lambda > 0,\ a \in \conv S \right\}.
\]

\begin{dfn}\label{def:nconus}
The {\emph{Frech\'et normal cone}} (in short, the \emph{normal cone}) to a set 
$A \subset \Red$ at a point $a_0\in A$ is the set of vectors $v\in\Red$ such 
that for any  $\varepsilon > 0$, there is  $\delta>0$ such that
 $\iprod{v}{a - a_0}  \leq \epsilon {\enorm{a - a_0}} \text{ for all } 
 a\in A\cap \parenth{ \delta\ball{d} + a_0}$. In short,
\[
\nfcone{A}{a_0} = \left\{ v \in \Red \st 
\iprod{v}{a - a_0}  \leq \littleo{\enorm{a - a_0}} \text{ for all }a\in A \right\}.
\]
Clearly, $\nfcone{A}{a_0}$ is a closed convex cone in $\Red$.
\end{dfn}

It is easy to see that the normal cone to a convex set $K$ in $\Red$ at a boundary point $a_0\in\bd K$ coincides with the usual normal cone, that is,
\begin{equation}\label{eq:normal_cone_dual_to_poscone}
\nfcone{K}{a_0}=
\left\{ v \in \Red \st 
\iprod{v}{a - a_0}  \leq 0 \text{ for all } a\in K\right\}
 =  \parenth{\poscone{K - a_0}}^\polar. 
\end{equation}

Mostly, we will consider the normal cone of the lifting of an upper semi-continuous log-concave function $f$. In particular, we will show in the next section that $\nfcone{\lifting{f}}{\upthing{u}}$ is not empty at any 
$\upthing{u}\in\essgraph{f}$.

\subsection{Contact pairs}
An important technical step in our analysis is to consider most but not all contact pairs of $f$ and $g$.
\begin{dfn}\label{defn:contactset}
For two functions $f,g\colon\Red\to\Re$, we call the set
		\[\contactset{f}{g} = 
		\big\{(\upthing{u},\upthing{v})\in \contactsetnr{f}{g} \st f(u)\neq 0 \big\} \cup
		\big\{(\upthing{u},\upthing{v})\in \contactsetnr{f}{g} \st f(u)=g(u)=0, \upthing{v}=(v, 0) \big\}.
		\]
their \emph{reduced set of contact pairs}, where
		\[\contactsetnr{f}{g} = 
		\big\{(\upthing{u},\upthing{v})\in\Redp\times\Redp \st 
		 \upthing{u}=(u, f(u)), u \in \cl{\supp {f}}\cap\cl{\supp {g}},\,  f(u)  = g(u),
		 \]
		 \[
		 \upthing{v} \in \nfcone{\lifting f}{\upthing{u}}\cap \nfcone{\lifting g}{\upthing{u}}, 
		 \, \iprod{\upthing{v}}{\upthing{u}}=1\big\},
		\]
as given in \Href{Definition}{defn:contactset_intro}.
\end{dfn}

The difference between $\contactsetnr{f}{g}$
 and $\contactset{f}{g}$ is that in the latter, we exclude outer normals with non-zero last coordinate at  contact points at which both functions vanish.

Note that if $g\leq f$ and $u\in\Red$ is such that $f(u)=g(u)$, then 
$\nfcone{\lifting{f}}{\upthing{u}}\subseteq\nfcone{\lifting{g}}{\upthing{u}}$, 
where $\upthing{u}=(u,f(u))$.
Thus, if $g \leq f$, then one has
		\[
		\contactsetnr{f}{g} = 
		\big\{(\upthing{u},\upthing{v})\in\Redp\times\Redp \st 
		 \upthing{u}=(u, f(u)) \in \essgraph{f} \cap  \essgraph{g},
		 \]
		 \[
		 \upthing{v} \in \nfcone{\lifting f}{\upthing{u}}, \quad
		 \, \iprod{\upthing{v}}{\upthing{u}}=1\big\}.
		\]

Also, we will need to ensure that for any point $\upthing{u} \in \essgraph{f}$ and any 
$\upthing{v} \in \nfcone{\lifting{f}}{\upthing{u}},$ the angle between $\upthing{u}$ and $\upthing{v}$ is acute, that is, $\iprod{\upthing{u}}{\upthing{v}} >0$.
In the case of convex sets, this condition easily follows from the assumption that the origin is in the interior of the set. In our functional case, a bit more care is needed. 
\begin{dfn}\label{dfn:weaklustarlike}
Let  $f \colon \Red \to [0, + \infty)$ be a function. 
We say that a set ${U} \subset \Red$    is a \emph{\wlocstar}$f$, if 
for every $ u \in  U \cap \supp f$, we have $\iprod{(u,f(u))}{\upthing{v}} > 0$ for all non-zero 
$\upthing{v}\in\nfcone{\lifting{f}}{(u, f(u))}.$ 
\end{dfn}

Since the Fr\'echet normal cone is always closed, we immediately have the following.
\begin{lem}\label{lem:locstar_geom_meaning}
Let  $f \colon \Red \to [0, + \infty)$ be a proper log-concave function, and $u \in \supp {f}$. Set $\upthing{u}= (u, f(u))$, and assume that $\{u\}$ is a \wlocstar$f$. Then 
\[
\nfcone{\lifting{f}}{\upthing{u}} = \poscone{\upthing{v} \in \nfcone{\lifting{f}}{\upthing{u}} \st \iprod{\upthing{v}}{\upthing{u}} = 1}.
\]
\end{lem}

\section{The normal cone of the lifting}\label{sec:normalcones}
In this section, we collect several technical statements about normal cones, which will be used in 
the proofs of the main results The proofs of the statements are based on mostly standard methods of real and convex analyses. This section is self-contained, that is, no proof herein relies on any out-of-section statement.
It may be ideal to omit these proofs on a first reading.

We recall that the \emph{Hausdorff distance}  between two compact subsets  $K$ and $L$ of $\Red$
is defined by 
\[
\hausmetric{K}{L} = \inf \left\{\lambda > 0 \st K \subset L + \lambda \ball{d};
\quad   L \subset H + \lambda \ball{d} \right\}.
\]

\begin{prop}[Hemicontinuity of the normal cone for convex sets]\label{prop:normalconesemicont}
Let $\{K_i\}_{i \in \N}$ be a sequence of bounded convex sets in $\Red$  converging in the Hausdorff distance to a convex set $K$, and let $\{u_i\}_{i \in \N}$  be a sequence of points with $u_i \in {K_i}$ converging to a point $u\in \bd{K}$. 
Let  $\{v_i\}_{i \in \N}$ be a sequence of outer normals $v_i\in\nfcone{K_i}{u_i}$ converging to a unit vector $v\in\Red$.
Then $v\in\nfcone{K}{u}$.
\end{prop}
\begin{proof}
According to equation (2.3) of \cite{schneider2014convex}, 
\[
 \nfcone{K}{u}=p_K^{-1}(u)-u
\]
for any $u\in K$, where $p_K:\Red\to K$ is the \emph{metric projection} onto $K$, that is, $p_K(x)$ is the unique point of $K$ that is closest to $x$.
Furthermore, Lemma~1.8.11 of \cite{schneider2014convex}, and the discussion preceding it state that the mapping $(K,x)\mapsto p_K(x)$ is continuous in both arguments.

Suppose for a contradiction that $v$ is not in $\nfcone{K}{u}$, that is, 
$p_{K}(u + v)\neq u$. Then, by the continuity of $(K,x)\mapsto p_K(x)$ in the second variable, there is a neighborhood $U$ of $u$, and a neighborhood $V$ of $u + v$ such that $p_{K}(V)\cap U = \emptyset$.
In turn, by the continuity of $(K,x)\mapsto p_K(x)$ in the first variable, there is a neighborhood $U^{\prime}$ of $u$ inside $U$ such that $p_{K}(V)\cap U^{\prime}=\emptyset$ for all sufficiently large $i \in \N$. This clearly contradicts the assumptions.
\end{proof}

By the symmetry of $\lifting{f}$ about $\Red$ and the convexity of $\supp{f}$, one obtains the following.
\begin{lem}[The normal cone of $\lifting{f}$ at the boundary of $\supp{f}$]\label{lem:nfcone_at_zero}
Let $f \colon \Red \to [0, \infty)$ be an upper semi-continuous log-concave function
and let  $u \in \bd{\supp f}$. Set $\upthing{u} = (u, f(u))$. Then 
\[
\projond \parenth{\nfcone{\lifting{f}}{\upthing{u}}} = 
\nfcone{\lifting{f}}{\upthing{u}} \cap \Red = 
\nfcone{\supp{f}}{u},
\]
where the last normal cone is considered in $\Red$, and $\projond:\Redp\to\Red$ denotes the orthogonal projection onto the first $d$ coordinates.
\end{lem}

The following simple lemma describes locally the normal cone of $\lifting{f}$ in terms of the normal cone of $\epi(-\ln f)$.

\begin{lem}[The normal cone of $\lifting{f}$ in $\supp{f}$]\label{lem:nfcone}
Let $f= e^{-\psi}  \colon \Red \to [0, \infty)$ be an upper semi-continuous log-concave function. 
Fix a point $u \in \supp f$, a scalar $\nu \in \R$, and set $\upthing{u} = (u, f(u))$.
Then 
\begin{align}\label{eq:normalconecomaprison}
 \parenth{v,\nu}\in\nfcone{\epi\psi}{(u,\psi(u))} \text{ if and only if }
 \parenth{v,\frac{-\nu}{f(u)}}\in\nfcone{\lifting{f}}{(u,f(u))}.
 \end{align}
 Furthermore,
 \begin{equation}\label{eq:normalconelevelset}
 \projond\left(\nfcone{\lifting{f}}{(u,f(u))}\right)\subseteq \nfcone{[f\geq f(u)]}{u}, 
 \end{equation}
where $\projond:\Redp\to\Red$ denotes the orthogonal projection onto the first $d$ coordinates.
\end{lem}
\begin{proof}
Clearly, $\nu$ cannot be positive. 
If $\nu=0$, then by \Href{Lemma}{lem:nfcone_at_zero},
the leftmost inclusion is equivalent to $v \in \nfcone{\dom \psi}{u}$,
and the rightmost is equivalent to $v \in \nfcone{\supp f}{u},$ which are 
equivalent by \eqref{eq:suppequalsdom}.

We thus assume that $\nu<0$. Since $\psi$ is a convex function, 
$\parenth{v,\nu}\in\nfcone{\epi\psi}{(u,\psi(u))}$ holds if and only if
\begin{equation*}
 \psi(x)\geq\psi(u)+\iprod{u-x}{\frac{v}{\nu}} 
\end{equation*}
for all $x \in \R^d$, which yields by exponentiation
\[
 f(x)= e^{-\psi(x)} \leq f(u)e^{\iprod{x-u}{\frac{v}{\nu}}}=
 f(u)\left[1+\iprod{x-u}{\frac{v}{\nu}}\right]+\littleo{|x-u|}.
\]
The latter, by the definition of the normal cone, is equivalent to
\begin{equation}\label{eq:vnuinnfcone}
  \parenth{v,\frac{-\nu}{f(u)}}\in\nfcone{\lifting{f}}{(u,f(u))}, 
\end{equation}
completing the proof of one implication in \eqref{eq:normalconecomaprison}. 

For the other direction, assume \eqref{eq:vnuinnfcone}, that is,
\[
 f(x)\leq f(u)\left[1+\iprod{x-u}{\frac{v}{\nu}}\right]+\littleo{|x-u|}.
\]
Since $1+t<e^t$ for all $t$, we have $f(x)\leq 
f(u)e^{\iprod{x-u}{\frac{v}{\nu}}}+\littleo{|x-u|}=
f(u)e^{\iprod{x-u}{\frac{v}{\nu}}+\littleo{|x-u|}}$, where we used $f(u)>0$. By 
taking logarithm, we obtain 
$\psi(x)\geq\psi(u)+\iprod{u-x}{\frac{v}{\nu}}+\littleo{|x-u|}$, and hence, 
$\parenth{v,\nu}\in\nfcone{\epi\psi}{(u,\psi(u)}$.

Equation \eqref{eq:normalconelevelset} is  a direct consequence of the definition of the normal cone and can be easily shown to hold for any function without the assumption of log-concavity. The proof of \Href{Lemma}{lem:nfcone} is complete.
\end{proof}

Since the normal cone to a convex subset of $\R^d$ at 
any point of its boundary contains non-zero vectors, the two previous lemmas yield the following.
\begin{cor}[The normal cone is not empty]\label{cor:normalcone_nonempty}
Let $f \colon \Red \to [0, \infty)$ be an upper semi-continuous log-concave function.
The normal cone to $\lifting{f}$ at any point of 
$\essgraph{f}$ contains non-zero vectors.
\end{cor}

\begin{lem}[Regularity of the normal cone of $\lifting{f}$]\label{lem:nfconeregularity}
Let $f \colon \Red \to [0, \infty)$ be an upper semi-continuous log-concave function. 
Fix  $\epsilon>0$ and a point $u \in \bd{\supp f}$ with $f(u)=0$.
Then there is a $\delta>0$ such that for every 
${u}_1 \in \bd{\supp{f}} \cap \parenth{\delta\ball{d} +u}$ and every 
 $\parenth{v_1,\nu_1}\in\nfcone{\lifting{f}}{({u}_1, f(u_1))}$, there is a
 $v\in\nfcone{\cl{\supp f}}{u}$ with $\iprod{\frac{v_1}{\enorm{v_1}}}{\frac{v}{\enorm{v}}}>1-\epsilon$.
 \end{lem}
\begin{proof}
Observe that if 
$\parenth{v_1,\nu_1}\in\nfcone{\lifting{f}}{(u_1, f(u_1))}$ for some point 
$u_1\in\cl{\supp{f}}$, then $v_1\in\nfcone{[f\geq f(u_1)]}{u_1}$. Thus, if $\delta$ is sufficiently small, then $f(u_1)$ is close to zero, and hence, $[f\geq f(u_1)] \cap\parenth{\ball{d} + u}$ is close to $\cl{\supp{f}} \cap \parenth {\ball{d} + u}.$ 
By applying \Href{Proposition}{prop:normalconesemicont}, we complete the proof.
\end{proof}

\Href{Lemma}{lem:nfcone} gives a description of the normal cone of $\lifting{f}$ which yields the following 
local description of $\lifting{f}$ in terms of curves.

\begin{lem}[Local description of $\lifting{f}$ at a non-zero point in terms of a curve]\label{lem:curve_in_the_lifting}
Let $f= e^{-\psi}  \colon \Red \to [0, \infty)$ be an upper semi-continuous log-concave function. 
Let $u \in  {\supp f}$ and set $\upthing{u} = (u, f(u))$. 
Let $\upthing{\xi}(t)  \colon [0,1] \to \R^{d+1}$ be a curve such that
$\upthing{\xi}(0) = (u, f(u))$ and the right derivative $\upthing{\xi}^{\prime}(0)$ at zero exists.
Consider the following statements:
\begin{enumerate}[(a)]
\item\label{item:curve_linearized_in_primal}
\[\upthing{\xi}^{\prime}(0) \in 
\interior \parenth{\nfcone{\lifting{f}}{\upthing{\xi}(0) }}^{\circ}.
\]
\item\label{item:curve_inclusion} There is a positive $\epsilon$ such that 
\[
\upthing{\xi}(t)\in \lifting{f}
\;\;\;\text{ for all } t \in [0, \epsilon].
\] 
\item\label{item:curve_linearized_in_primalweak} 
\[
\upthing{\xi}^{\prime}(0)\in 
 \parenth{\nfcone{\lifting{f}}{\upthing{\xi}(0) }}^{\circ}.
\]
\end{enumerate}
 Then \eqref{item:curve_linearized_in_primal} implies \eqref{item:curve_inclusion}, and \eqref{item:curve_inclusion} implies \eqref{item:curve_linearized_in_primalweak}.
\end{lem}

\begin{proof}
With the identification $\Redp=\Red\oplus\Re$, we split the coordinates of $\upthing{\xi}(t)$ as $\upthing{\xi}(t)= \parenth{\xi(t),\mu(t)}$ that is, $\xi\colon[0,1]\to\Red$ and $\mu\colon[0,1]\to\Re$ with $\xi(0)=u$ and $\mu(0)=f(u)$.

Clearly, if $\lifting{f}$ was a convex set in $\Redp$, then the statement would follow from basic properties of supporting hyperplanes to convex sets. We will use the fact that even though $\lifting{f}$ is not convex in general, but the epigraph of $\psi$ is. The proof relies on translating the question on the behavior of $\upthing{\xi}$ with respect to $\lifting{f}$ to a question concerning the behavior of the curve
 \[
  \upthing{\eta}(t)=\parenth{\xi(t),-\log \mu(t)}
 \]
with respect to $\epi{\psi}$.
 
Since $f$ is upper semi-continuous and log-concave, $\epi{\psi}$ is a closed convex set in $\Redp$ with nonempty interior. It follows that 
\begin{equation}\label{eq:curve_linearized_in_primal_eta}
   \upthing{\eta}^{\prime}(0)\in
   \interior \parenth{\nfcone{\epi\psi}{\upthing{\eta}(0) }}^{\circ}
\end{equation}
implies
\begin{equation}\label{eq:curve_inclusion_eta}
 \upthing{\eta}(t)\in\epi{\psi}\text{ for all }t\in[0,\varepsilon], \text{ with some }\varepsilon>0,
\end{equation}
which, in turn, implies
\begin{equation}\label{eq:curve_linearized_in_primalweak_eta}
   \upthing{\eta}^{\prime}(0)\in
  \parenth{\nfcone{\epi\psi}{\upthing{\eta}(0) }}^{\circ}.
\end{equation}

Clearly, statement \eqref{item:curve_inclusion} of the lemma is equivalent to 
\eqref{eq:curve_inclusion_eta}, thus, in order to prove the lemma, we need to 
show the equivalence of inclusions \eqref{item:curve_linearized_in_primal} and 
\eqref{eq:curve_linearized_in_primal_eta}, and the equivalence of inclusions 
\eqref{item:curve_linearized_in_primalweak} and 
\eqref{eq:curve_linearized_in_primalweak_eta}.

Using
$\upthing{\eta}^{\prime}(0)=\parenth{\xi^{\prime}(0), \frac{-\mu^{\prime}(0)}{f(u)}}$
and the definition of the polar cone, we have
\[
 \upthing{\eta}^{\prime}(0)\in
 \parenth{\nfcone{\epi{\psi}}{\upthing{\eta}(0) }}^{\circ} \Longleftrightarrow
\iprod{\parenth{\xi^{\prime}(0), \frac{-\mu^{\prime}(0)}{f(u)}}}{\parenth{v,\nu}}\leq 0
\text{ for all }\parenth{v,\nu}\in\nfcone{\epi{\psi}}{\upthing{\eta}(0)}.
\]
The latter, by \eqref{eq:normalconecomaprison}, is equivalent to
\[ 
\iprod{\parenth{\xi^{\prime}(0), \mu^{\prime}(0)}}{\parenth{v,\frac{-\nu}{f(u)}}}\leq 0
\text{ for all }\parenth{v,\frac{-\nu}{f(u)}}\in\nfcone{\lifting{f}}{\upthing{\xi}(0)},
\]
which, in turn, is equivalent to 
\[
 \upthing{\xi}^{\prime}(0)\in 
 \parenth{\nfcone{\lifting{f}}{\upthing{\xi}(0) }}^{\circ}.
\]

In summary, \eqref{eq:curve_linearized_in_primalweak_eta} is equivalent to statement \eqref{item:curve_linearized_in_primalweak} of the lemma.
The equivalence of \eqref{eq:curve_linearized_in_primal_eta} and
statement \eqref{item:curve_linearized_in_primal} is shown the same way.
\end{proof}

Since $\cl{\supp{f}}=\lifting{f}\cap\Red$ is a closed convex set in $\Red$, we have the following statement.
\begin{lem}[Local description of $\lifting{f}$ at a zero point in terms of a horizontal curve]\label{lem:flat_curve_in_the_lifting}
Let $f  \colon \Red \to [0, \infty)$ be an upper semi-continuous log-concave function. 
Let $u \in  \cl{\supp f}$  such that $f(u)= 0.$ Denote $\upthing{u} = (u, f(u))=(u,0)$.  
Let $\upthing{\xi}(t)  \colon [0,1] \to \Red$ be a curve such that
$\upthing{\xi}(0) = (u, 0)$ and the right derivative $\upthing{\xi}^{\prime}(0)$ at zero exists.
Consider the following statements:
\begin{enumerate}[(a)]
\item\label{item:flat_curve_linearized_in_primal}
\[\upthing{\xi}^{\prime}(0) \in 
\interior \parenth{\nfcone{\supp{f}}{\upthing{\xi}(0) }}^{\circ}.
\]
\item\label{item:flat_curve_inclusion} There is a positive $\epsilon$ such that 
\[
\upthing{\xi}(t)\in \lifting{f}
\;\;\;\text{ for all } t \in [0, \epsilon].
\] 
\item\label{item:flat_curve_linearized_in_primalweak} 
\[
\upthing{\xi}^{\prime}(0)\in 
 \parenth{\nfcone{\supp{f}}{\upthing{\xi}(0) }}^{\circ}.
\]
\end{enumerate}
Then the implications \eqref{item:curve_linearized_in_primal} $\Rightarrow$ \eqref{item:curve_inclusion} $\Rightarrow$ \eqref{item:curve_linearized_in_primalweak} hold, where polarity is meant in $\Red$.
\end{lem}

\section{John's problem}\label{sec:john}

Fix $s > 0$, and two functions $f, g  \colon \Red \to [0, \infty)$. 
In addition to John $s$-problem \eqref{eq:john_problem_intro}, we  will consider the following optimization problem. 
\medskip

\textbf{Positive position John $s$-problem for $f$ and $g$:} Find
\begin{equation}\label{eq:john_problem_pos}
\max\limits_{h \in \funppos{g} } 
	\int_{\Red} h^s 
	\quad \text{subject to} \quad
	h \leq f.
\end{equation}
\medskip 

We say that $g$ is a \emph{global maximizer} in the (Positive position) John $s$-problem, if 
for any $(A\oplus\alpha,a)$ in $\MM$ (resp., in $\MM^+$), we have that $\int_{\Red} h^s \leq \int_{\Red} g^s$ whenever $h\leq f$ and $h(x)=\alpha h(Ax+a)$. 
On the other hand, $g$ is a \emph{local maximizer} in the (Positive position) John $s$-problem, if there is a neighborhood $\mathcal U$ of $( \id_d\oplus 1, 0)$ in $\MM$ (resp., in $\MM^+$) such that for any $(A\oplus\alpha,a)\in\mathcal U$, we have that $\int_{\Red} h^s \leq \int_{\Red} g^s$ whenever $h\leq f$ and $h(x)=\alpha h(Ax+a)$.

We are ready to state our first main result, a general version of 
\Href{Theorem}{thm:john_intro}. The assumptions of the theorem are quite 
technical, they will be explained in 
\Href{Section}{sec:corollaries_and_discussion}, where we discuss natural 
situations when they hold.

\begin{thm}[John's condition]\label{thm:john_condition_general}
Fix $s > 0$. Let $f\colon\Red \to [0,+\infty)$ be a proper log-concave function,
and let $\gbase \colon \Red \to [0, \infty)$ be an upper semi-continuous function 
such that 
\begin{itemize}
\item $\logconc{\gbase}$ satisfies our Basic Assumptions (see page~\pageref{assumptions:basic}); 
\item the set of contact points $\contactpoint{f}{\gbase}$ is a \wlocstar ${f}$;
\item the reduced set of contact pairs $\contactset{f}{\gbase}$ is bounded. 
\end{itemize} 
Then setting $g =\logconc{\gbase}$, the following hold:
	\begin{enumerate}
		\item\label{item:local-maximum-implies-glmp}
			If $h=g$  is  a local maximizer in John $s$-problem \eqref{eq:john_problem_intro} for $f$ and $g$,
			then there exist contact pairs 
			$(\upthing{u}_1,\upthing{v}_1)$, $\dots$,  
			$(\upthing{u}_m,\upthing{v}_m)$ $\in \contactset{f}{\gbase}$
			and positive weights $c_1,\dots,c_m$ such that
		\begin{equation}\label{eq:functional_glmp}
		\sum_{i=1}^{m} c_i {u}_i \otimes {v}_i = 
		 \id_{d}, \quad 
		 \sum_{i=1}^{m} c_i f(u_i)\nu_i =  s
					\quad\text{and}\quad
					\sum_{i=1}^{m} c_i v_i=0,
				\end{equation}
        where $\upthing{u}_i=(u_i, f(u_i))$ and $\upthing{v}_i=(v_i,\nu_i)$.
		\item\label{item:glmp-implies-global-maximum}
			If there exist contact pairs and positive weights  satisfying  equation \eqref{eq:functional_glmp},
			then $g$ is a global maximizer in Positive position John $s$-problem \eqref{eq:john_problem_pos} for $f$ and $g$.
	\end{enumerate}
\end{thm}

We emphasize that in this result, we maximize the integral of the log-concave 
envelope of a given function $\gbase$, 
but we use the contact pairs of $f$ and $\gbase$. It allows us to consider the case when the integral of $\gbase$ is zero, for example, if $\gbase$ has a finite number of non-zero values.  A similar results for not necessarily log-concave function $\gbase$ immediately follows from \Href{Theorem}{thm:john_condition_general}, because a position of $g$ is below $f$ if and only if the corresponding position of $\gbase$ is below $f$ by the log-concavity of $f$.
\begin{cor}
 Let the functions $f,\gbase \colon\Red\to[0,+\infty)$ satisfy the assumptions of \Href{Theorem}{thm:john_condition_general}. In addition, let  the integral of $\gbase$ be positive. Then 
 	\begin{enumerate}
		\item 
			If $h=\gbase$  is  a local maximizer in  John $s$-problem \eqref{eq:john_problem_intro} for $f$ and $\gbase$,
			then there exist contact pairs 
			$(\upthing{u}_1,\upthing{v}_1), \dots,  
			(\upthing{u}_m,\upthing{v}_m) \in \contactset{f}{\gbase}$
			and positive weights $c_1,\ldots,c_m$ satisfying \eqref{eq:functional_glmp}.
		\item
			If there exist contact pairs and positive weights  satisfying  equation \eqref{eq:functional_glmp},
			then $h=\gbase$ is a global maximizer in  Positive position John $s$-problem \eqref{eq:john_problem_pos} for $f$ and $\gbase$.
	\end{enumerate}
\end{cor}

In \Href{Section}{subsec:sufficient_conditions}, we will discuss conditions on $\gbase$ which guarantee that the conditions of \Href{Theorem}{thm:john_condition_general} on $f$ and $\gbase$ hold with \emph{any} proper log-concave function $f$.

\subsection{Strategy of the proof of 
\texorpdfstring{\Href{Theorem}{thm:john_condition_general}}{
Theorem~\ref{thm:john_condition_general}}}

\begin{dfn}\label{dfn:contactop}
For any $(\upthing{u},\upthing{v})\in\Redp\times\Redp$,  
we define the \emph{John-type extended contact operator} by
\[
\contactopjohn{\upthing{u}}{\upthing{v}} = \parenth{\parenth{u \otimes v} \oplus \mu\nu,v}\in \MMbar,
\]
where $\upthing{u}=(u,\mu)\in\Red\oplus\Re=\Redp$, and $\upthing{v}=(v,\nu)\in\Red\oplus\Re=\Redp$.

For two functions $f,g:\Red\to[0,\infty)$, we denote the reduced set of John-type extended contact operators by
 \begin{equation}\label{eq:johncontactoperators}
 \contactopsetjohn{f}{g}=
 \big\{\contactopjohn{\upthing{u}}{\upthing{v}} \st
 (\upthing{u},\upthing{v})\in \contactset{f}{g} \big\}\subset\MMbar.
 \end{equation}
\end{dfn}

We break up the proof of \Href{Theorem}{thm:john_condition_general} into several steps. 
First, we will use a geometric reformulation of the equations in 
\eqref{eq:functional_glmp}. It is easy to see that those equations encode the 
fact that the point $\parenth{\id_{d} \oplus s, 0} \in \MM$ is in the positive 
cone of the set 
$\contactopsetjohn{f}{\gbase}$ in $\MMbar$. It then follows that if no set of 
contact pairs satisfies \eqref{eq:functional_glmp}, then there is a linear 
hyperplane separating the point $\parenth{\id_{d} \oplus s, 0}$ and the set 
$\contactopsetjohn{f}{\gbase}$. 
 
We then aim at turning this separation of $\parenth{\id_{d} \oplus s, 0}$ and 
the set $\contactopsetjohn{f}{\gbase}$ into a perturbation of the function $g$ 
which is of greater integral than $g$, and is still pointwise below $f$. 
However, to obtain this perturbation of $g$, we need strong separation, which poses an 
important difficulty. To solve it, we need to show that 
$\contactopsetjohn{f}{\gbase}$ is compact in $\MMbar$. To that end, we will show 
that the boundedness of the set of contact pairs $\contactset{f}{\gbase}$ yields 
the compactness of $\contactopsetjohn{f}{\gbase}$. Note that we will not 
investigate the boundedness of $\contactset{f}{\gbase}$ itself in the current 
section, it will be addressed in \Href{Section}{sec:boundedcontactpairs}.

Thus, to prove \eqref{item:local-maximum-implies-glmp} of \Href{Theorem}{thm:john_condition_general}, we first write strong separation analytically: if no contact pairs yield \eqref{eq:functional_glmp}, then there is an $(H\oplus\gamma,h)\in\MMbar$ (a normal to the separating hyperplane) such that
\begin{equation}\label{eq:separation_explanatory_part}
	\iprod{\parenth{H \oplus \gamma,h}}
	{\parenth{\sid,0}} > 0
	\quad\text{and}\quad
	\iprod{\parenth{H \oplus \gamma,h}}
	{A_J} < 0
\end{equation}
for all $A_J \in \contactopsetjohn{f}{\gbase}$.

To obtain the needed perturbation of $g$, we construct a certain kind of 
``average'' of two functions below a given log-concave function $f$ so that the 
new function remains below $f$. This will be  a straight forward adjustment of averaging 
two positions of a set inside a given convex set. Using this averaging 
construction, we will construct a curve $\Gamma_t, t \in [0, \tau]$  in $\MM$ 
starting at $\parenth{\id_{d+1},0}$ such that its directional vector at zero is precisely the normal vector $ 
\parenth{H \oplus \gamma,h}$. Positions of $g$ correspond to the points of 
$\Gamma_t$ in a natural way: $\alpha g\!\parenth{A^{-1}(x-a)}$ corresponds to 
$\parenth{A \oplus \alpha, a} \in \MM$.  That is, the position corresponding to 
$\Gamma_0=\parenth{\id_{d+1}, 0}$ is $g$ itself, and $\Gamma_t, t \in [0,\tau]$ can be seen as a 
homotopy of $\lifting{g}$ for a  sufficiently small positive $\tau.$ 

The next step is to see what properties of the curve $\Gamma_t$, $t \in [0, \tau]$,
should possess in order to guarantee the inequalities in 
\eqref{eq:separation_explanatory_part}. It turns out that the leftmost 
inequality in \eqref{eq:separation_explanatory_part} essentially means that the 
integral of a position of $g$ corresponding to a point $\Gamma_t$ is  greater 
than the integral of $g$ itself for all sufficiently small  $t.$ 
\begin{remark}
We note that the comparison of integrals appears only at this step. So using our approach, one might maximize more sophisticated functionals than the $L_s$-norm of $g$. 
\end{remark}

Finally, we will show that the rightmost inequality in \eqref{eq:separation_explanatory_part} essentially means that  the position of $g$ corresponding to $\Gamma_t$ remains below $f$ for all sufficiently small $t$.

To put everything together, in the proof of necessary condition \eqref{item:local-maximum-implies-glmp}, we will assume that 
the point $\parenth{\id_{d} \oplus s, 0}$ is strongly separated from the set $\contactopsetjohn{f}{\gbase}$ in $\MMbar$. Then, we will construct a curve $\Gamma_{t}$ in $\MMbar$ using the normal vector of the separating hyperplane, and after that, we will show that it defines a homotopy of $\lifting{g}$ with the desired property. In the proof of sufficient condition \eqref{item:glmp-implies-global-maximum}, we will assume that $g$ is not the global maximizer, and we will construct a curve $\Gamma_t$ in $\MMbar$ using our averaging construction and then, we will show that the directional vector of $\Gamma_t$ is the normal vector of a hyperplane that separates $\parenth{\id_{d} \oplus s, 0}$ from $\contactopsetjohn{f}{\gbase}$ in $\MMbar$.

\subsection{Main components of the proof of \texorpdfstring{\Href{Theorem}{thm:john_condition_general}}{Theorem~\ref{thm:john_condition_general}}}
By a routine compactness argument and by \Href{Corollary}{cor:normalcone_nonempty}, on has the following. 
\begin{lem}
 Let  functions $f,\gbase\colon\Red\to[0,+\infty)$ satisfy the assumptions of \Href{Theorem}{thm:john_condition_general}. Assume  $h=g$  is  a local maximizer in John $s$-problem \eqref{eq:john_problem_intro}.  Then the sets $\contactset{f}{\gbase}$ and $\contactpoint{f}{\gbase}$ are non-empty.
\end{lem}

Let us show that the boundedness of the set of contact pairs yields the compactness of the set of contact operators.
\begin{lem}[Compactness of the set of contact operators]\label{lem:compact_contact_operator}
 Let  functions $f,\gbase\colon\Red\to[0,+\infty)$ satisfy the assumptions of \Href{Theorem}{thm:john_condition_general}. Then $\contactopsetjohn{f}{\gbase}$ is a compact subset of $\MMbar$.
\end{lem}

\begin{proof}

The definition of $\contactopsetjohn{f}{\gbase}$ and the boundedness of
$\contactset{f}{\gbase}$ imply that $\contactopsetjohn{f}{\gbase}$ is bounded.

Let us show that $\contactopsetjohn{f}{\gbase}$ is closed. 
Consider a sequence $\{(\upthing{u}_i, \upthing{v}_i)\} \subset \contactset{f}{\gbase}$, where $\upthing{u}_i = (u_i, \mu_i)$ and $\upthing{v}_i = (v_i, \nu_i),$ such that $\contactopjohn{\upthing{u}_i}{\upthing{v}_i}$ is convergent.
Since $\contactset{f}{\gbase}$ is bounded, we may assume by passing to a subsequence that 
$\lim\limits_{i \to \infty} \upthing{u}_i = \upthing{u} = (u, \mu)$ and 
$\lim\limits_{i \to \infty} \upthing{v}_i = \upthing{v} = \upthing{v} = (v, \nu)$.
By upper semi-continuity, $(u, \mu) \in \essgraph{f} \cap \essgraph{\gbase}.$

If $\mu > 0$, then the  convex function $-\ln f$ is finite  in some open neighborhood of $u$ in its effective domain.  Using \Href{Lemma}{lem:nfcone} and then applying
\Href{Proposition}{prop:normalconesemicont} to the compact convex set 
$K = \{(x, h) \st -\ln f(x) \leq h \leq  - \ln f(u) +1 \}$  yield
\[
\upthing{v} \in \nfcone{\lifting{f}}{\upthing{u}}.
\]

Now, consider the case $\mu = 0$.
Since the set $\{\nu_i \st i\in\Ze^+\}$ is bounded and 
$\lim\limits_{i \to \infty} \mu_i = 0$, we have
\[
\lim\limits_{i \to \infty} \mu_i  \nu_i = 0 = \mu\nu.
\]
Hence, $\iprod{\upthing{v}}{\upthing{u}} = \iprod{v}{u}.$ 
By this and by \Href{Lemma}{lem:nfcone_at_zero}, it suffices to show that $v \in \nfcone{\supp f}{u} \subset \Red$. 

By the assumptions of \Href{Theorem}{thm:john_condition_general}, the origin is in the interior of $\supp{f}$. Consider the sequence of compact convex  sets 
$\left[ f \geq f(u_i) \right] \cap 2\enorm{u}\ball{d}.$  This sequence converges to $\cl{\supp{f}} \cap (2\enorm{u}\ball{d})$. By \eqref{eq:normalconelevelset},
 $v_i \in \nfcone{\left[ f \geq f(u_i) \right] \cap 2\enorm{u}\ball{d}}{u_i}.$ Thus, using \Href{Lemma} {lem:nfconeregularity}, we conclude  $\lim\limits_{i \to \infty} v_i= v$ 
 belongs to $\nfcone{\supp f}{u}.$  
  Consequently, $\contactopsetjohn{f}{\gbase}$ is a closed bounded set of the finite-dimensional vector space $\MMbar$ and hence, it is compact. 
\end{proof}

Next, we reformulate equation \eqref{eq:functional_glmp} in terms of separation of a closed convex set from a point in the finite dimensional real vector space $\MMbar$.

\begin{lem}[Separation of operators]\label{lem:separation_John_problem}
 Let the functions $f,\gbase \colon\Red\to[0,+\infty)$ satisfy the assumptions of \Href{Theorem}{thm:john_condition_general}. Then the following assertions are equivalent:
\begin{enumerate}
\item\label{item:nocontacts}
There are no contact pairs of $f$ and $\gbase$ and positive weights  satisfying  equation \eqref{eq:functional_glmp}.
\item\label{item:strictseparation}
There exists $ \parenth{H \oplus \gamma, h}\in\MMbar$ such that
		\begin{equation}\label{eq:john_s-concave-stict-separation}
			\iprod{\parenth{H \oplus \gamma,h}}
			{\parenth{\sid,0}} > 0
			\quad\text{and}\quad
			\iprod{\parenth{H \oplus \gamma,h}}
			{\contactopjohn{\upthing{u}}{\upthing{v}}} < 0
		\end{equation}
		for all $(\upthing{u},\upthing{v})\in \contactset{f}{\gbase}$.
\item\label{item:weakseparation}
There exists $ \parenth{H \oplus \gamma, h}\in\MMbar$ such that
		\begin{equation}\label{eq:john_s-concave-nonstrict-separation}
			\iprod{\parenth{H \oplus \gamma,h}}
			{\parenth{\sid,0}} > 0
			\quad\text{and}\quad
			\iprod{\parenth{H \oplus \gamma,h}}
			{\contactopjohn{\upthing{u}}{\upthing{v}}} \leq 0
		\end{equation}
		for all $(\upthing{u},\upthing{v})\in \contactset{f}{\gbase}$.
\end{enumerate}
\end{lem}

\begin{proof}[Proof of \Href{Lemma}{lem:separation_John_problem}]
For any $(\upthing{u},\upthing{v})\in\contactset{f}{\gbase}$, we have $\iprod{\upthing{u}}{\upthing{v}}=1$ thus, $\tr{\parenth{u\otimes v}}+f(u)\nu=1$. Taking trace in the first equation in \eqref{eq:functional_glmp} and adding it to the second equation therein yields that $\sum_{i=1}^m c_i=d+s$. 
It follows that assertion \eqref{item:nocontacts} of the lemma is equivalent to the assertion that $\frac{1}{d+s}\parenth{\sid,0}$ is not in the convex hull of
the set $\contactopsetjohn{f}{\gbase}$ of John-type extended contact operators defined by \eqref{eq:johncontactoperators}.
By \Href{Lemma}{lem:compact_contact_operator}, $\contactopsetjohn{f}{\gbase}$ is compact. 
Note that both the set $\contactopsetjohn{f}{\gbase}$ and the point  $\frac{1}{d+s}\parenth{\sid,0}$ belong to the affine hyperplane  
$
\left\{
(\mathcal{A},a) \in \MMbar \st \tr{\mathcal{A}} = 1
\right\}.
$
Thus, the statements in \Href{Lemma}{lem:separation_John_problem} are reduced to the separation of a point from a compact subset $\contactopsetjohn{f}{\gbase}$ by a linear hyperplane in the finite-dimensional real vector space $\MMbar$, a basic notion in convexity.
\end{proof}

Second, the following lemma allows us to interpolate between two functions below a given one in such a way that the new function remains  below the given one.
\begin{lem}[Inner interpolation of functions]\label{lem:inner-function-interpolation}
	Let $f \colon \Red \to [0, +\infty)$ be a log-concave function and  $g \colon \Red \to [0, +\infty)$ be a function.
	Let $\alpha_1,\alpha_2>0$, $A_1, A_2$ be non-singular matrices of order $d,$
	and  $a_1, a_2\in\Red$
	be such that
		\[
			\alpha_1 g(A_1^{-1} (x - a_1)) \leq f(x)
			\quad \text{and}\quad
			\alpha_2 g(A_2^{-1} (x - a_2)) \leq f(x).
		\]
	for all $x\in\Red$.
	Let $\beta_1,\beta_2>0$ be such that $\beta_1+\beta_2=1$.
	Define
		\[
			\alpha = \alpha_1^{\beta_1} \alpha_2^{\beta_2},
			\quad
			A = \beta_1 A_1 + \beta_2 A_2,
			\quad\text{and}\quad
			a = \beta_1 a_1 + \beta_2 a_2.
		\]
	Assume that $A$ is non-singular.	
	Then
		\begin{equation}\label{eq:inner-function-interpolation-1}
			\alpha g(A^{-1} (x - a)) \leq f(x).
		\end{equation}
	If $A_1$ and $A_2$ are positive definite, and $g$ is integrable, then we also have
		\begin{equation}\label{eq:inner-function-interpolation-2}
			\int_{\Red} \alpha g(A^{-1} (x - a)) \di x \geq
				\left(\int_{\Red}\alpha_1 g(A_1^{-1} (x - a_1) ) \di x\right)^{\beta_1}
				\left(\int_{\Red}\alpha_2 g(A_2^{-1} (x - a_2) ) \di x\right)^{\beta_2}
		\end{equation}
with equality if and only if $A_1 = A_2.$
\end{lem}
\begin{proof}
Fix $x \in \Red$ and define
\[
x_1 =  A_1 A^{-1} x, \quad x_2 =  A_2 A^{-1} x. 
\]
By assumption of the lemma, 
\begin{equation}\label{eq:inner_interp_height}
f(x_1 + a_1) \ge  \alpha_1 g \! \parenth{A^{-1}x}
\quad \text{and} \quad 
f(x_2 + a_2) \ge  \alpha_2 g \! \parenth{A^{-1}x}.
\end{equation}
By our definitions, $ \beta_1 (x_1 + a_1) + \beta_2(x_2 + a_2)  = x + a.$
Therefore,  by the log-concavity of $f$,
\[
f(x + a)    \ge f^{\beta_1}( x_1 + a_1) f^{\beta_2}( x_2 + a_2),
 \]
which, by \eqref{eq:inner_interp_height}, yields
\[
{f(x + a)}  \ge 
   \alpha g \! \parenth{A^{-1}x}.
\]
Inequality \eqref{eq:inner-function-interpolation-1} follows.
Inequality \eqref{eq:inner-function-interpolation-2} immediately follows from  Minkowski's determinant inequality \eqref{eq:minkowski_det_multipl_ineq}.
\end{proof}

Next, we observe that the leftmost inequalities in  \eqref{eq:john_s-concave-stict-separation} and \eqref{eq:john_s-concave-nonstrict-separation} compare the integrals of $g$ and a perturbation of $g$ defined by
 \begin{equation}\label{eq:perturbationdef}
g_t (x) =  \alpha_t g \! \parenth{A^{-1}_t(x - a_t)}, \text{ for } x\in\Red.  
 \end{equation}

\begin{lem}[Integral of a perturbation of $g$]\label{lem:john_separation_integral}
Fix $s > 0$, and let $g \colon \Red \to [0, +\infty)$ be a function such that $g^s$ is of finite positive integral.
	 Let $\Gamma(t) = (A_t \oplus \alpha_t, a_t),$ $ t \in [0,1]$ be a  curve in $\mathcal{M}$ with $(A_0 \oplus \alpha_0, a_0)=(\id_{d+1},0)$, and assume that the right derivative of $\Gamma$ at $t=0$ is of the form $(H\oplus\gamma,h)$.  Define  the perturbation ${g}_t$ of $g$ by \eqref{eq:perturbationdef}. Consider the following statements:
 \begin{enumerate}[(a)]
 \item\label{item:int_admissible_linearized_in_primal} 
\[			
 	\iprod{\parenth{H \oplus \gamma,h}}{\parenth{\sid,0}} >  0;
\] 
 \item\label{item:int_gammaadmissible}  	
 		\[
				\int {g}_t^s>\int g^s 
			\]
		for all $t \in (0, \tau]$ and  some $\tau > 0;$	
 \item\label{item:int_admissible_linearized_in_primalweak} 
\[		
 	\iprod{\parenth{H \oplus \gamma,h}}
			{\parenth{\sid,0}} \geq  0.
\] 
 \end{enumerate}
  Then \eqref{item:int_admissible_linearized_in_primal} implies \eqref{item:int_gammaadmissible}, and \eqref{item:gammaadmissible} implies \eqref{item:int_admissible_linearized_in_primalweak}.
\end{lem}
\begin{proof}
One has
\[
	\int_{\Red} {g}_t^s =  
	 \alpha^{s} \det A_t \int_{\Red} g^s = 
	\parenth{1 + t \gamma + \littleo{t}}^{s} \parenth{1 +  t\tr{H} + \littleo{t} } \int_{\Red} g^s=
\]
\[
=\parenth{1 + t \parenth{s \gamma + \tr{H}} + \littleo{t} } \int_{\Red} g^s
=\parenth{1 + t \iprod{\parenth{H \oplus \gamma,h}}
			{\parenth{\sid,0}} + \littleo{t} } \int_{\Red} g^s,
\]
and the statement follows.
\end{proof}

The following lemma is an exercise in compactness.
\begin{lem}[Homothopy of a compact set]\label{lem:most_general_homotopy_inclusion}
Let $F$ be a closed set in $\Red$ and $K \subset F$ be a non-empty compact set.
Let $\mathcal{H} \colon K \times [0, \tau] \to \Red$ be a homotopy between $K$ and some set $K_\tau$ such that for every $u \in \bd F \cap K$ and some positive $\epsilon_u,$ 
the curve $\mathcal{H} (u, t), t \in [0, \epsilon_u]$ belongs to $F.$ Then  there is positive $\epsilon$ such that
$\mathcal{H} (K, t) \subset F$ for all $t \in [0, \epsilon].$ 
\end{lem}
\begin{proof}
Fix any  $u \in K$, and consider the continuous function
$t\mapsto\operatorname{dist}(\mathcal{H} (u, t), F)$. It is zero at $t=0$. Clearly, if $u$ belongs to the interior of $F$, then this function is zero on a proper interval $t \in [0, \epsilon_u]$. By the assumption of the lemma, the same holds if $u \in \bd F \cap K$. 
The compactness of $K$ yields the assertion of the lemma.
\end{proof}

Finally, we show that the rightmost inequalities in \eqref{eq:john_s-concave-stict-separation} and \eqref{eq:john_s-concave-nonstrict-separation} encode that a certain perturbation of $g$ is pointwise below $f$.

\begin{thm}[Characterization of admissible perturbations]\label{thm:admissiblelinearized_john}
 Let the functions $f,g\colon\Red\to[0,+\infty)$ satisfy the assumptions of \Href{Theorem}{thm:john_condition_general}, and let $\Gamma(t) = (A_t \oplus \alpha_t, a_t), t \in [0,1]$ be a curve in $\mathcal{M}$ with $(A_0 \oplus \alpha_0, a_0)=(\id_{d+1},0)$, and assume that the right derivative of $\Gamma$ at $t=0$ is of the form $(H\oplus\gamma,h)$.  Define  the perturbation ${g}_t$ of ${g}$ by \eqref{eq:perturbationdef}.
 Consider the following statements:
 \begin{enumerate}[(a)]
  \item\label{item:admissible_linearized_in_primal}
$			\iprod{\parenth{H \oplus \gamma,h}}
			{\contactopjohn{\upthing{u}}{\upthing{v}}} < 0
$
		for all $(\upthing{u},\upthing{v})\in \contactset{f}{\gbase}$.
  \item\label{item:gammaadmissible} 
  There is $\epsilon > 0$ such that ${g}_t  \leq {f}$ for all $t\in[0,\epsilon]$.
  \item\label{item:admissible_linearized_in_primalweak}
$			\iprod{\parenth{H \oplus \gamma,h}}
			{\contactopjohn{\upthing{u}}{\upthing{v}}} \leq 0
$
		for all $(\upthing{u},\upthing{v})\in \contactset{f}{\gbase}$.
 \end{enumerate}
Then \eqref{item:admissible_linearized_in_primal} implies \eqref{item:gammaadmissible}, and \eqref{item:gammaadmissible} implies \eqref{item:admissible_linearized_in_primalweak}.
\end{thm}
\begin{proof}
By the log-concavity of $f$, a position of $g$ is below $f$ if and only if the corresponding position of $\gbase$ is below $f.$ Hence, it suffices to consider the perturbation of $\gbase$ given by 
\[
\tilde{g}_t (x) =  \alpha_t \gbase \! \parenth{A^{-1}_t(x - a_t)}, \text{ for } x\in\Red.
\]

Since our sets are symmetric about $\Red$, it suffices to consider the sets $\essgraph{\tilde{g}_t}$. 
Define the homothopy
$\mathcal{H} \colon$ $\parenth{\essgraph{\gbase}} \times [0, 1] \to \R^{d+1}$ by
\[
\mathcal{H}(\upthing{y},t)=\parenth{A_t y + a_t,\   \alpha_t \gbase(y)}
\]
for all $y \in \cl{\supp \gbase}$ and $\upthing{y} = (y, \gbase(y)).$
That is,
\[
\mathcal{H}\! \parenth{\essgraph{\gbase},t} = \essgraph{\tilde{g}_t}.
\]

Consider an arbitrary  $u \in \cl{\supp \gbase}$ and 
$\upthing{v}=(v,\nu)\in\Red \times \Re$. Set $\upthing{u} = (u, \gbase(u))$.
One has
\[
 \mathcal{H}^\prime:=\derivativeatzero\mathcal{H}(\upthing{u}, t)=\parenth{ Hu + h,\ \gamma \gbase(u)},
\]
and thus, by \Href{Definition}{dfn:contactop},
\begin{equation}\label{eq:iprodvhiprodtensor}
\iprod{\upthing{v}}{\mathcal{H}^\prime}=\iprod{v}{Hu + h} +  
\gamma \nu \gbase(u)= 
\iprod{\big(H\oplus\gamma,h\big)}{\contactopjohn{\upthing{u}}{\upthing{v}}}.
\end{equation}

Consider the case  $\gbase(u)>0$.
We recall the assumption of \Href{Theorem}{thm:john_condition_general} according to which 
$\contactpoint{f}{g_b}$ is a \wlocstar $f$. 
Using \Href{Lemma}{lem:locstar_geom_meaning} in identity \eqref{eq:iprodvhiprodtensor}, 
we see that for a fixed $\upthing{u} \in  \essgraph{f}$ with non-zero last coordinate, assertion \eqref{item:admissible_linearized_in_primal} of the theorem  is equivalent to
\[
\mathcal{H}^\prime \in \interior \parenth{\nfcone{\lifting{f}}{\upthing{u}}}^{\circ}.
\]
Similarly, assertion \eqref{item:admissible_linearized_in_primalweak} of \Href{Theorem}{thm:admissiblelinearized_john} is equivalent to
\[
\mathcal{H}^\prime \in \parenth{\nfcone{\lifting{f}}{\upthing{u}}}^{\circ}.
\]

Consider the case  $\gbase(u)=0$.
Since $\supp \logconc{g}$ is bounded and contains the origin in its interior by the assumption of \Href{Theorem}{thm:john_condition_general},  the origin is contained in the interior of $\supp f.$  This ensures that at every contact point $\upthing{u}\in \lifting{\gbase}\cap\bd{\lifting{f}}$ with $f(u) = 0$, every outer normal direction to 
$\supp{f}$ has an acute angle with $\upthing{u} = (u,0) \in \Red$, and hence, it can be represented by a vector $\upthing{v}= (v,0)\in \R^d$ such that $\iprod{\upthing{u}}{\upthing{v}}=1$, which yields $(\upthing{u},\upthing{v})\in\contactset{\lifting{\gbase}}{f}$.

It follows  from identity \eqref{eq:iprodvhiprodtensor}  that for a fixed $\upthing{u}$ whose  last coordinate is zero, assertion \eqref{item:admissible_linearized_in_primal} of the theorem  is equivalent to
\[
\mathcal{H}^\prime \in \interior \parenth{\nfcone{\supp{f}}{\upthing{u}}}^{\circ},
\]
where all the sets and the polarity are meant in $\Red.$
Similarly, assertion \eqref{item:admissible_linearized_in_primalweak} of \Href{Theorem}{thm:admissiblelinearized_john} is equivalent to
\[
\mathcal{H}^\prime \in \parenth{\nfcone{\supp{f}}{\upthing{u}}}^{\circ}.
\]

Consequently, the implication \eqref{item:curve_inclusion} $\Rightarrow$ \eqref{item:curve_linearized_in_primalweak} of \Href{Lemma}{lem:curve_in_the_lifting} and the implication \eqref{item:flat_curve_inclusion} $\Rightarrow$ \eqref{item:flat_curve_linearized_in_primalweak} of \Href{Lemma}{lem:flat_curve_in_the_lifting} yield \eqref{item:gammaadmissible} $\Rightarrow$ \eqref{item:admissible_linearized_in_primalweak} in the theorem.

 To obtain the implication \eqref{item:admissible_linearized_in_primal} $\Rightarrow$ \eqref{item:gammaadmissible} of \Href{Theorem}{thm:admissiblelinearized_john}, we will use  \Href{Lemma}{lem:most_general_homotopy_inclusion} with the roles $K=\essgraph{\gbase}$ and $F=\lifting{f}.$  We need to check that the curve
 $\xi(t) = \mathcal{H}\parenth{\upthing{u},t}$ remains in $\lifting{f}$ for 
 an arbitrary fixed $\upthing{u} \in \essgraph{f} \cap \essgraph{g_b}$ and all sufficiently small  $t.$ 
The implication \eqref{item:curve_linearized_in_primal} $\Rightarrow$ \eqref{item:curve_inclusion} of \Href{Lemma}{lem:curve_in_the_lifting}  
and the implication \eqref{item:flat_curve_linearized_in_primal} $\Rightarrow$ \eqref{item:flat_curve_inclusion} of \Href{Lemma}{lem:flat_curve_in_the_lifting} yield this property of the curve $\xi(t)$ in the corresponding cases. 
This completes the proof of \Href{Theorem}{thm:admissiblelinearized_john}.
\end{proof}
\begin{remark}
\Href{Theorem}{thm:admissiblelinearized_john}.
is the only place in the proof of \Href{Theorem}{thm:john_condition_general} at which we 
use the assumption that $\contactpoint{f}{\gbase}$ is a \wlocstar$f$. 
\end{remark}
\subsection{Proof of Theorem~\ref{thm:john_condition_general}}
We start with assertion \eqref{item:local-maximum-implies-glmp} of the theorem.
Assume  that there are no contact pairs and positive weights  satisfying  equation~\eqref{eq:functional_glmp}.
Then, by \Href{Lemma}{lem:separation_John_problem}, there exists $ \parenth{H \oplus \gamma, h} \in \MMbar$ such that
\begin{equation}\label{eq:john_s-concave-separation}
    \iprod{\parenth{H \oplus \gamma,h}}
    {\parenth{\sid,0}} > 0
    \quad\text{and}\quad
    \iprod{\parenth{H \oplus \gamma,h}}
    {\contactopjohn{\upthing{u}}{\upthing{v}}} < 0
\end{equation}
for all $\parenth{\upthing{u},\upthing{v}} \in \contactset{f}{\gbase}$.

The matrix 
\(
\id_{d} + t {H}
\)
 is non-singular for all $t \in [0, \tau]$ for some positive $\tau$.  
We use $\parenth{H \oplus \gamma,h}$ to obtain a perturbation of $g$ within the 
class $\funpos{g}$ which is pointwise below $f$ but is of larger integral than  
$g$. Set
\[
    g_t(x)=(1+ t\gamma) g\! \parenth{\parenth{\id_d + t H}^{-1}(x-th)}.
\]
Clearly,  $g_t \in \funpos{g}$  
for all $t \in [0, \tau]$. 
By  implication \eqref{item:int_admissible_linearized_in_primal} $\Rightarrow$  \eqref{item:int_gammaadmissible} of \Href{Lemma}{lem:john_separation_integral}, 
    \[
        \int g_t^s > \int g^s
    \]
for all $t \in (0, \tau]$ with some $\tau > 0.$
Using the implication \eqref{item:admissible_linearized_in_primal} $\Rightarrow$  \eqref{item:gammaadmissible} of  \Href{Theorem}{thm:admissiblelinearized_john},
one obtains that $g_t  \leq f$ for all sufficiently small $t$. 
Thus, $g$ is not a local maximizer in the John $s$-problem \eqref{eq:john_problem_intro}, completing the proof of assertion~(\ref{item:local-maximum-implies-glmp}) of \Href{Theorem}{thm:john_condition_general}.

We proceed with assertion~(\ref{item:glmp-implies-global-maximum}) of the theorem.
Assume that $h=g$ is not a global maximizer in  Positive position John $s$-problem \eqref{eq:john_problem_pos}.
That is,  there exist a positive definite matrix $A$, $\gamma \in \R$ and $h \in \Red$  such that 
the function $g_1$ defined by
\[
g_1(x) = e^{\gamma} g\parenth{A^{-1}(x-h)}, \text{ for } x\in\Red
\]
satisfies $g_1 \leq  f$ and $\int_{\Red} g_1^s > \int_{\Red} g^s$. We may assume a bit more by applying a slight contraction on 
$\lifting{g_1}$ (which is a transformation within $\funppos{g}$): 
$\lifting{g_1}\subset\interior\parenth{\lifting{f}}$. We will use $g_1$ to define a perturbation $g_t$ of $g$ (a curve in $\MM$) that is below $f$ with larger integral, and then, taking the derivative of that perturbation at the starting point $t=0$, we will obtain  $\parenth{H\oplus\gamma,h}\in\MMbar$ that separates $\parenth{\sid, 0}$ from the set $\contactopsetjohn{f}{g}$, which by \Href{Lemma}{lem:separation_John_problem} will yield the assertion.

Set $H = A - \id_d$.
We begin by showing that 
\begin{equation}\label{eq:john_suff_integral_separation}
\iprod{\parenth{H \oplus \gamma, h}}{\parenth{\sid, 0}} > 0.
\end{equation}

Indeed, since $\lifting{g_1}$ is a  compact subset of $\interior\parenth{\lifting{f}}$,  there is a $\delta>0$ such that $A - \delta \id_d$ is positive definite and the function $\tilde{g}_1$  defined by 
\[
\tilde g_1(x) = e^{\gamma} g_1\parenth{\parenth{A - \delta \id_d}^{-1}(x-h)},
\]
satisfies the relations $\tilde{g}_1\leq f$ and $\int_{\Red} g^s < \int_{\Red} \tilde{g}_1^s$.

\[
\tilde{g}_t (x) =  e^{t\gamma} g\parenth{\parenth{\id_d + t\parenth{H - \delta \id_d}}^{-1}(x - th)} \in 
\funppos{g}.
\] 		
By \Href{Lemma}{lem:inner-function-interpolation}, 
$\tilde{g}_t \leq f$ for all $t \in [0,1],$ 	and
\[
\int_{\Red} {\tilde{g}_t^s} \geq
    \parenth{ \int_{\Red} {g}^s}^{1-t}
        \parenth{\int_{\Red} \tilde{g}_1^s}^{t}
    \geq \int_{\Red} {g^s}.
\]
Using implication \eqref{item:int_gammaadmissible} $\Rightarrow$ \eqref{item:int_admissible_linearized_in_primalweak} of \Href{Lemma}{lem:john_separation_integral}, we have
\[
 0 \leq \iprod{\parenth{\parenth{H - \delta \id_d} \oplus \gamma, h}}{\parenth{\sid, 0}}  = 
\iprod{\parenth{H  \oplus \gamma, h}}{\parenth{\sid, 0}} - \delta d, 
\]
yielding
\begin{equation*}
\iprod{\parenth{H \oplus \gamma, h}}{\parenth{\sid, 0}} \geq \delta d > 0,
\end{equation*}
and \eqref{eq:john_suff_integral_separation} follows.

For $t \in [0,1)$, define 
\[
{g}_t (x) =  e^{t\gamma} g\parenth{\parenth{\id_d + t H}^{-1}(x - th)}.
\] 		
Clearly, ${g}_t  \in \funppos{g}$ for all $t \in [0,1)$. 
By the choice of $g_1$ and by \Href{Lemma}{lem:inner-function-interpolation}, one sees that ${g}_t \leq f$ for all $t \in [0,1].$
The implication \eqref{item:gammaadmissible} $\Rightarrow$  \eqref{item:admissible_linearized_in_primalweak} in \Href{Theorem}{thm:admissiblelinearized_john} implies that
$\iprod{\parenth{H \oplus \gamma,h}}
			{\contactopjohn{\upthing{u}}{\upthing{v}}}$ $\leq 0$
		for all $\parenth{\upthing{u},\upthing{v}} \in \contactset{f}{\gbase}$.
Combining this with inequality \eqref{eq:john_suff_integral_separation},
\Href{Lemma}{lem:separation_John_problem} yields assertion \eqref{item:glmp-implies-global-maximum} of \Href{Theorem}{thm:john_condition_general}, whose proof is thus complete.

\section{L\"owner problem}\label{sec:lowner}
Fix $s > 0$, and two functions $f, g  \colon \Red \to [0, \infty)$. 

In addition to L\"owner $s$-problem \eqref{eq:lowner_problem_intro},
we will consider the following optimization problem. 
\medskip

\textbf{Positive position L\"owner $s$-problem for $f$ and $g$:} Find
\begin{equation}
\label{eq:lowner_problem_pos}
\min\limits_{h \in \funppos{g} }
	\int_{\Red} h^s 
	\quad \text{subject to} \quad
	f \leq h.
\end{equation}

We define local and global minimizers to these problems in the same way as maximizers to the two John $s$-problems introduced in \Href{Section}{sec:john}.

The next theorem, our main, most general result concerning the L\"owner $s$-problem provides a condition of optimality in terms of the polars of the functions -- for the definition, see \Href{Section}{sec:polarity}.

\begin{thm}[L\"owner's condition]\label{thm:lowner_condition_general}
Fix $s>0.$ Let $f\colon\Red\to[0,+\infty)$ be a proper function.
and let $g \colon \Red \to [0, \infty)$ be a proper log-concave function with $f \leq g.$
Assume that 
 \begin{itemize}
\item $\loglego{g}$ satisfies our Basic Assumptions (see page~\pageref{assumptions:basic}); 
\item the set of contact points $\contactpoint{\loglego{f}}{\loglego{g}}$ is a \wlocstar$\loglego{f}$;
\item the set $\contactset{\loglego{f}}{\loglego{g}}$ is bounded. 
\end{itemize} 
Then the following hold.
	\begin{enumerate}
		\item\label{item:lowner-local-minimum-implies-glmp}
			If $h=g$ is a local minimizer in  L\"owner $s$-problem \eqref{eq:lowner_problem_intro} for $f$ and $g$,
			then there exist contact pairs 
			$(\upthing{u}_1 ,\upthing{v}_1), \dots,  
			(\upthing{u}_m ,\upthing{v}_m) \in \contactset{\loglego{f}}{\loglego{g}}$
			and positive weights $c_1,\ldots,c_m$ such that
		\begin{equation}\label{eq:functional_glmp-lowner}
		\sum_{i=1}^{m} c_i {v}_i \otimes {u}_i = 
		 \id_{d}, \quad 
		 \sum_{i=1}^{m} c_i \loglego{g}(u_i ) \cdot \nu_i =  s
					\quad\text{and}\quad
					\sum_{i=1}^{m} c_i\loglego{g}(u_i ) \cdot \nu_i u_i =0,
        \end{equation} 
        where $\upthing{u}_i=(u_i, \loglego{g}(u_i ))$ and 
        $\upthing{v}_i=(v_i,\nu_i) \in \nfcone{\lifting{\loglego{g}}}{\upthing{u}_i}$.
		\item\label{item:glmp-lowner-implies-global-minimum}
			If there exist contact pairs and positive weights satisfying equation \eqref{eq:functional_glmp-lowner},
			then $h=g$ is a global maximizer in Positive position L\"owner $s$-problem \eqref{eq:lowner_problem_pos} for $f$ and $g$.
	\end{enumerate}
\end{thm}


The proof is similar to that of \Href{Theorem}{thm:john_condition_general} with an essential additional idea. Instead of studying the inequality $f \leq g_t$ for a perturbation $g_t$ of $g$, we will study the equivalent inequality $\loglego{f} \geq \loglego{g_t}$. 

\subsection{Main components of the proof of \texorpdfstring{\Href{Theorem}{thm:lowner_condition_general}}{Theorem~\ref{thm:lowner_condition_general}}}
First, the following lemma, the dual to 
\Href{Lemma}{lem:inner-function-interpolation}, allows us to interpolate between 
two functions above a given one in such a way that the new function remains 
above the given one.

\begin{lem}[Outer interpolation of functions]\label{lem:outer-interpolation}
	Let $f \colon \Red \to [0, +\infty)$ be a  function and  $g \colon \Red \to [0, +\infty)$ be a log-concave function.
	Let $\alpha_1,\alpha_2>0$, $A_1, A_2$ be non-singular matrices of order $d,$
	and  $a_1, a_2\in\Red$
	be such that
		\[
			f(x) \leq \alpha_1 g(A_1x +a_1)
			\quad \text{and}\quad
			f(x) \leq \alpha_2 g(A_2x +a_2)
		\]
	for all $x\in\Red$.
	Let $\beta_1,\beta_2>0$ be such that $\beta_1+\beta_2=1$.
	Define
		\[
			\alpha = \alpha_1^{\beta_1} \alpha_2^{\beta_2},
			\quad
			A = \beta_1 A_1 + \beta_2 A_2,
			\quad\text{and}\quad
			a = \beta_1 a_1 + \beta_2 a_2.
		\]
	Assume that $A$ is non-singular.	
	Then
		\begin{equation}\label{eq:outer-interpolation-1}
			f(x) \leq \alpha g(Ax + a).
		\end{equation}
	If $A_1$ and $A_2$ are positive definite and $g$ is integrable, then also
		\begin{equation}\label{eq:outer-interpolation-2}
			\int_{\Red} \alpha g(Ax+a) \di x \leq
				\left(\int_{\Red} \alpha_1 g(A_1x + a_1) \di x\right)^{\beta_1}
				\left(\int_{\Red} \alpha_2 g(A_2x + a_2) \di x\right)^{\beta_2}
		\end{equation}
with equality if and only if $A_1 = A_2.$
\end{lem}
\begin{proof}
By the assumption of the lemma,
\[
f(x)  = f^{\beta_1}(x) \cdot f^{\beta_2}(x) \leq 
\alpha_1^{\beta_1} g^{\beta_1}(A_1 x + a_1) \cdot 
\alpha_2^{\beta_2} g^{\beta_2}(A_2 x + a_2) = 
\alpha g^{\beta_1}(A_1 x + a_1) \cdot 
 g^{\beta_2}(A_2 x + a_2)
\]
By our definitions, $ \beta_1 (A_1 x + a_1) + \beta_2(A_2 x + a_2)  = Ax + a.$
Thus,  the log-concavity of $g$ yields inequality \eqref{eq:outer-interpolation-1}.

Inequality \eqref{eq:outer-interpolation-2} immediately follows from  Minkowski's determinant inequality \eqref{eq:minkowski_det_multipl_ineq} and its equality condition.
\end{proof}

\begin{dfn}\label{dfn:contactoplowner}
For any $(\upthing{u},\upthing{v})\in\Redp\times\Redp$,  
we define the \emph{L\"owner-type extended contact operator} by
\[
\contactoplowner{\upthing{u}}{\upthing{v}} = \parenth{\parenth{v \otimes u} \oplus \mu\nu, \mu\nu u}\in \MMbar,
\]
where $\upthing{u}=(u,\mu)\in\Redp$, and $\upthing{v}=(v,\nu)\in\Redp$.

For two functions $f,g:\Red\to[0,\infty)$, we denote the set of L\"owner-type extended contact operators by
 \begin{equation}\label{eq:lownercontactoperators}
 \contactopsetlowner{f}{g}=
 \big\{\contactoplowner{\upthing{u}}{\upthing{v}} \st
 (\upthing{u},\upthing{v})\in \contactset{f}{g} \big\}\subset\MMbar.
 \end{equation}
\end{dfn}

\begin{lem}[Compactness of the set of contact operators]\label{lem:compact_contact_operatorLowner}
 Let  functions $f,g\colon\Red\to[0,+\infty)$ satisfy the assumptions of \Href{Theorem}{thm:lowner_condition_general}. Then 
 $\contactopsetlowner{\loglego{f}}{\loglego{g}}$ is a compact subset of $\MMbar$.
\end{lem}
We omit the proof as it is essentially the same as the proof of \Href{Lemma}{lem:compact_contact_operator}.


Next, we reformulate equation \eqref{eq:functional_glmp-lowner}  in terms of separation of a closed convex set from a point in the finite-dimensional real vector space $\MMbar$. The proof is identical to the proof of \Href{Lemma}{lem:separation_John_problem}, we omit it.

\begin{lem}[Separation of operators]\label{lem:separation_Lowner_problem}
For any two upper semi-continuous functions $f,g\colon\Red\to[0,+\infty)$, the following assertions are equivalent:
\begin{enumerate}
\item  There are no contact pairs of $f$ and $g$ and positive weights satisfying equation \eqref{eq:functional_glmp-lowner}.
\item There exists $ \parenth{H \oplus \gamma, h}\in\MMbar$ such that
\begin{equation}\label{eq:lowner-stict-separation}
    \iprod{\parenth{H \oplus \gamma,h}}
    {\parenth{\sid,0}} > 0
    \quad\text{and}\quad
\iprod{\big(H\oplus\gamma,h\big)}{\contactoplowner{\upthing{u} }{\upthing{v}}}
< 0
\end{equation}
for all $\parenth{\upthing{u}, \upthing{v}} \in \contactset{\loglego{f}}{\loglego{g}}$.
\item There exists $ \parenth{H \oplus \gamma, h}\in\MMbar$ such that
\begin{equation}\label{eq:lowner-nonstrict-separation}
    \iprod{\parenth{H \oplus \gamma,h}}
    {\parenth{\sid,0}} > 0
    \quad\text{and}\quad
\iprod{\big(H\oplus\gamma,h\big)}{\contactoplowner{\upthing{u} }{\upthing{v}}}
 \leq 0
\end{equation}
for all $\parenth{\upthing{u}, \upthing{v}} \in \contactset{\loglego{f}}{\loglego{g}}$.
\end{enumerate}
\end{lem}

Next, we observe that the leftmost inequalities in  \eqref{eq:lowner-stict-separation} and \eqref{eq:lowner-nonstrict-separation} compare the integrals of $g$ and a perturbation of $g$ defined by
 \begin{equation}\label{eq:perturbationdefLowner}
g_t (x) =  \frac{1}{\alpha_t} g \! \parenth{\transpose{A_t}x  + a_t },\;\text{ for }x\in\Red. 
 \end{equation}

\begin{lem}[Integral of a perturbation of $g$]\label{lem:lowner_separation_integral}
Fix $s > 0$, and let $g \colon \Red \to [0, +\infty)$ be a function such that $g^s$ is of finite positive integral.
	 Let $\Gamma(t) = (A_t \oplus \alpha_t, a_t),$ $ t \in [0,1]$ be a  curve in $\mathcal{M}$ with $(A_0 \oplus \alpha_0, a_0)=(\id_{d+1},0)$, and assume that the right derivative of $\Gamma$ at $t=0$ is of the form $(H\oplus\gamma,h)$.  Define the perturbation $g_t$ of $g$ by \eqref{eq:perturbationdefLowner}. Consider the following statements:
 \begin{enumerate}[(a)]
 \item\label{item:int_lowner_admissible_linearized_in_primal}
  \[\iprod{\parenth{H \oplus \gamma,h}}{\parenth{\sid,0}} > 0;\] 
  \item\label{item:int_lowner_gammaadmissible}  	
 \[\int {g}_t^s  <  \int g^s\]
 for all $t \in (0, \tau]$ and  some $\tau > 0.$	
 \item\label{item:int_lowner_admissible_linearized_in_primalweak} 
 \[\iprod{\parenth{H \oplus \gamma,h}}{\parenth{\sid,0}} \geq  0.\] 
 \end{enumerate}
  Then \eqref{item:int_admissible_linearized_in_primal} implies \eqref{item:int_gammaadmissible}, and \eqref{item:gammaadmissible} implies \eqref{item:int_admissible_linearized_in_primalweak}.
\end{lem}
\begin{proof}
One has
\[
	\int_{\Red} {g}_t^s =  
	 \alpha_t^{-s} \det \transpose{\parenth{A_t^{-1}}} \int_{\Red} g^s = 
	\parenth{1 + t  \gamma + \littleo{t}}^{-s} \parenth{1 - t\tr{H} + \littleo{t} } \int_{\Red} g^s=
\]
\[
=\parenth{1 - t \parenth{s \gamma + \tr{H}} + \littleo{t} } \int_{\Red} g^s
=\parenth{1 - t \iprod{\parenth{H \oplus \gamma,h}}
			{\parenth{\sid,0}} + \littleo{t} } \int_{\Red} g^s.
\]
The result follows.
\end{proof}

\begin{lem}[Polar of a transformed function]\label{lem:polar_func_composed_affine}
Let $f \colon \Red \to [0, + \infty)$ be a proper log-concave function,
$A$ be a non-singular matrix of order $d$, $\alpha > 0$ and $a \in \Red$.
Set $\tilde{f}(x) = \alpha f(Ax +a)$.
Then
\[
\loglego{\tilde{f}}(y) = 
\frac{\loglego{f}(\transpose{\parenth{A^{-1}}}y)}{\alpha}  \cdot
e^{\iprod{\transpose{\parenth{A^{-1}}}y}{\ \! a}}. 
\] 
\end{lem}
\begin{proof}
\[
\loglego{\tilde{f}}(y) = 
\inf\limits_{x \in \supp \tilde{f}} \frac{e^{-\iprod{x}{y}}}{\tilde{f}(x)} =
\frac{1}{\alpha } \inf\limits_{x \in \supp \tilde{f}} \frac{e^{-\iprod{x}{y}}}{{f}(Ax +a)} =
\frac{1}{\alpha } 
\inf\limits_{x \in \supp \tilde{f}}
\frac{e^{-\iprod{Ax}{\ \transpose{\parenth{A^{-1}}}y}}}{{f}(Ax +a)}=
\]
\[ 
\frac{1}{\alpha } 
\inf\limits_{x \in \supp \tilde{f}}
\frac{e^{-\iprod{Ax + a}{\ \transpose{\parenth{A^{-1}}}y}}}{{f}(Ax +a)} 
e^{\iprod{a}{\ \transpose{\parenth{A^{-1}}}y}}=
 \frac{e^{\iprod{a}{\ \transpose{\parenth{A^{-1}}}y}}}{\alpha } 
\inf\limits_{z \in \supp {f}}
\frac{e^{-\iprod{z}{\ \transpose{\parenth{A^{-1}}}y}}}{{f}(z)} = 
\] 
\[
\frac{\loglego{f}(\transpose{\parenth{A^{-1}}}y)}{\alpha}  \cdot
e^{\iprod{\transpose{\parenth{A^{-1}}}y}{\ \! a}}.
\]
\end{proof}

Using \Href{Lemma}{lem:polar_func_composed_affine} and the compactness of $\lifting{\loglego{g}}$  together with 
 \Href{Corollary}{cor:normalcone_nonempty}, we obtain the following. 
\begin{lem}
 Let  functions $f,g\colon\Red\to[0,+\infty)$ satisfy the assumptions of \Href{Theorem}{thm:lowner_condition_general}. Assume  $h=g$  is  a local minimizer in L\"owner $s$-problem \eqref{eq:lowner_problem_intro}.  Then the sets $\contactset{\loglego{f}}{\loglego{g}}$ and $\contactpoint{\loglego{f}}{\loglego{g}}$ are non-empty.
\end{lem}

Finally, we show that the rightmost inequalities \eqref{eq:lowner-stict-separation} and \eqref{eq:lowner-nonstrict-separation} encode that a certain perturbation of $g$ is pointwise above $f$.

\begin{thm}[Characterization of admissible perturbations]\label{thm:admissiblelinearized_lowner}
 Let the functions $f,g\colon\Red\to[0,+\infty)$ satisfy the assumptions of \Href{Theorem}{thm:lowner_condition_general}, and let $\Gamma(t) = (A_t \oplus \alpha_t, a_t), t \in [0,1]$ be a  curve in $\mathcal{M}$ with $(A_0 \oplus \alpha_0, a_0)=(\id_{d+1},0)$, and assume that the right derivative of $\Gamma$ at $t=0$ is of the form $(H\oplus\gamma,h)$.  Define the perturbation $g_t$ of $g$ by \eqref{eq:perturbationdefLowner}. Consider the following statements:
 \begin{enumerate}[(a)]
  \item\label{item:lowner_admissible_linearized_in_primal}
$\iprod{\big(H\oplus\gamma,h\big)}{\contactoplowner{\upthing{u}}{\upthing{v}}}
< 0$
for all $\parenth{\upthing{u}, \upthing{v}} \in \contactset{\loglego{f}}{\loglego{g}}$.
  \item\label{item:lowner_admissible} 
  There is $\epsilon > 0$ such that $f \leq {g}_t$ for all $t\in[0,\epsilon]$.
  \item\label{item:lowner_admissible_linearized_in_primalweak}
$\iprod{\big(H\oplus\gamma,h\big)}{\contactoplowner{\upthing{u}}{\upthing{v}}}
\leq 0$
for all $\parenth{\upthing{u}, \upthing{v}} \in \contactset{\loglego{f}}{\loglego{g}}$.
 \end{enumerate}
Then \eqref{item:admissible_linearized_in_primal} implies \eqref{item:gammaadmissible}, and \eqref{item:gammaadmissible} implies \eqref{item:admissible_linearized_in_primalweak}.
\end{thm}
\begin{proof}
The relation $f \leq {g}_t$ is equivalent to $\loglego{g_t} \leq \loglego{f},$ 
which we will proceed to work with. Since $\lifting{\loglego{f}}$ and $\lifting{\loglego{g}}$ are symmetric about $\Red$, it suffices to consider 
$ \essgraph{\loglego{g_t}}.$ 
Define the homothopy $\mathcal{H} \colon 
\essgraph{\loglego{g}} \times [0, 1] \to \R^{d+1}$ by  
\[
\mathcal{H}(\upthing{y},t)=\parenth{{A_t} y ,\  \loglego{g_t} \! \parenth{{A_t}y}}. 
\]
for all $y \in \cl{\supp \loglego{g}}$ and $\upthing{y} = (y, \loglego{g}(y)).$
By \Href{Lemma}{lem:polar_func_composed_affine},
\[
\loglego{g_t} ({A_t} y) = 
\alpha_t \loglego{g} \! \parenth{{A_t^{-1}}A_t y} 
e^{\iprod{A_t^{-1} {A_t} y}{\ a_t}} = 
\alpha_t \loglego{g} (y) e^{\iprod{y}{a_t}}
\]
and
$
  {{A_t}} \parenth{\supp \loglego{g}}  = \supp \loglego{{g}_t}.
$ 
That is,
\[
\mathcal{H}\! \parenth{\essgraph{\loglego{g}},t} = \essgraph{\loglego{g_t}}.
\]

Consider an arbitrary  $u \in \cl{\supp \loglego{g}}$ and 
$\upthing{v}=(v,\nu)\in\Red \times \R$. Set $\upthing{u} = (u, \loglego{g}(u))$.
One has
\[
 \mathcal{H}^\prime:=\derivativeatzero\mathcal{H}(\upthing{u} , t)=
 \parenth{ Hu ,\ \loglego{g}(u) \parenth{\gamma + \iprod{u}{h}}},
\]
and thus, by \Href{Definition}{dfn:contactoplowner},
\begin{equation}\label{eq:iprodvhiprodtensor_lowner}
\iprod{\upthing{v}}{\mathcal{H}^\prime}=\iprod{v}{H u }
 +  \nu \loglego{g}(u) \gamma +  \nu \loglego{g}(u)  \iprod{u}{h}
=\iprod{\big(H\oplus\gamma,h\big)}{\contactoplowner{\upthing{u}}{\upthing{v}}}.
\end{equation}

Consider the case  $\loglego{g}(u)>0$. 
We recall the assumption of \Href{Theorem}{thm:lowner_condition_general} according to which 
 $\contactpoint{\loglego{f}}{\loglego{g}}$ is a \wlocstar$\loglego{f}$.
Using \Href{Lemma}{lem:locstar_geom_meaning} in identity \eqref{eq:iprodvhiprodtensor_lowner}, 
we see that for a fixed $\upthing{u} \in  \essgraph{\loglego{f}}$ with non-zero last coordinate, assertion \eqref{item:lowner_admissible_linearized_in_primal} of the theorem  is equivalent to
\[
\mathcal{H}^\prime \in \interior \parenth{\nfcone{{\lifting{\loglego{f}}}}{\upthing{u} }}^{\circ}.
\]
Similarly, assertion \eqref{item:admissible_linearized_in_primalweak} of \Href{Theorem}{thm:admissiblelinearized_john} is equivalent to
\[
\mathcal{H}^\prime \in\parenth{\nfcone{{\lifting{\loglego{f}}}}{\upthing{u} }}^{\circ}.
\]

Consider the case  $\loglego{g}(u)=0$.
Since $\supp \loglego{g}$ is bounded and contains the origin in its interior by the assumption of \Href{Theorem}{thm:lowner_condition_general},  the origin is contained in the interior of $\supp \loglego{f}$.  This ensures that at every contact point $\upthing{u}\in \essgraph{\loglego{g}}\cap \essgraph{\loglego{f}}$ with $\loglego{f}(u) = 0$, every outer normal direction to 
$\supp{\loglego{f}}$ has an acute angle with $\upthing{u} = (u,0) \in \Red$, and hence, it can be represented by a vector $\upthing{v}= (v,0)\in \R^d$ such that 
$\iprod{\upthing{u}}{\upthing{v}}=1$, which yields 
$(\upthing{u},\upthing{v})\in\contactset{\loglego{g}}{\loglego{f}}$.

It follows  from identity \eqref{eq:iprodvhiprodtensor}  that for a fixed $\upthing{u}$ whose  last coordinate is zero, assertion \eqref{item:admissible_linearized_in_primal} of the theorem  is equivalent to
\[
\mathcal{H}^\prime \in \interior \parenth{\nfcone{\supp{\loglego{f}}}{\upthing{u}}}^{\circ},
\]
where all the sets and the polarity are meant in $\Red.$
Similarly, assertion \eqref{item:admissible_linearized_in_primalweak} of \Href{Theorem}{thm:admissiblelinearized_john} is equivalent to
\[
\mathcal{H}^\prime \in \parenth{\nfcone{\supp{\loglego{f}}}{\upthing{u}}}^{\circ}.
\]

The rest of the proof of \Href{Theorem}{thm:admissiblelinearized_lowner} is identical to the end of the proof of \Href{Theorem}{thm:admissiblelinearized_john}, and so we omit it.
\end{proof}
\begin{remark}
\Href{Theorem}{thm:admissiblelinearized_lowner}.
is the only place in the proof of \Href{Theorem}{thm:lowner_condition_general} at which we 
use the assumption that $\contactpoint{\loglego{f}}{\loglego{g}}$ is a \wlocstar$\loglego{f}$. 
\end{remark}
	
\subsection{Proof of Theorem~\ref{thm:lowner_condition_general}}
We start with assertion \eqref{item:lowner-local-minimum-implies-glmp} of the theorem.
Assume that there are no contact pairs and positive weights satisfying equation~\eqref{eq:functional_glmp-lowner}.
Then, by \Href{Lemma}{lem:separation_Lowner_problem}, there exists $\parenth{H \oplus \gamma, h} \in \MMbar$ such that
\begin{equation}\label{eq:lowner-concave-separation}
    \iprod{\parenth{H \oplus \gamma,h}}
    {\parenth{\sid,0}} > 0
    \quad\text{and}\quad
    \iprod{\parenth{H \oplus \gamma,h}}
    {\contactoplowner{\upthing{u} + c }{\upthing{v}}} < 0
\end{equation}
for all $\parenth{\upthing{u}, \upthing{v}}\in \contactset{\loglego{f}}
{\loglego{g}}$.

The matrix 
\(
\id_{d} - t {H}
\)
is non-singular and $1 + t \gamma  > 0$ for all $t \in [0, \tau]$ for some positive $\tau.$ 
 Set
\( A_t = {{\parenth{\id_{d} - t {H}}^{-1}}}\) and
		\[
			g_t(x)=\frac{1}{1 +  t \gamma} \ g\! \parenth{\transpose{A_t} x  + t h}.
		\]
Clearly,  $g_t \in \funpos{g}$ for all $t \in [0, \tau].$ 
By  implication \eqref{item:int_lowner_admissible_linearized_in_primal} $\Rightarrow$  \eqref{item:int_lowner_gammaadmissible} of \Href{Lemma}{lem:lowner_separation_integral}, 
\[
    \int g_t^s < \int g^s
\]
for all $t \in (0, \tau]$ and  some $\tau > 0$.

Using the implication \eqref{item:lowner_admissible_linearized_in_primal} $\Rightarrow$  \eqref{item:lowner_admissible} of  \Href{Theorem}{thm:admissiblelinearized_lowner},
one gets that $f \leq g_t$ for all sufficiently small $t$. 
Thus, $g$ is not a local minimizer in  L\"owner $s$-problem \eqref{eq:lowner_problem_intro},
completing the proof of assertion \eqref{item:lowner-local-minimum-implies-glmp} of \Href{Theorem}{thm:lowner_condition_general}.

We proceed with assertion \eqref{item:glmp-implies-global-maximum} of the theorem.
Assume that $g$ is not a global minimizer in  Positive position L\"owner $s$-problem  \eqref{eq:lowner_problem_pos}. That is, there exist a positive definite matrix $A$,
$\gamma \in \R, h \in \Red,$  such that 
the function $g_1$ defined by
\[
g_1(x) = e^{-\gamma} g\! \parenth{A x + h}\;\text{ for }\; x\in\Red
\]
satisfies $f \leq  g_1$ and $\int_{\Red} g_1^s < \int_{\Red} g^s$.
We will use $g_1$ to define a perturbation $g_t$ of $g$ (a curve in $\MM$) that is above $f$ with smaller integral, and then, taking the derivative of that perturbation at the starting point $t=0$, we will find $\parenth{H\oplus\gamma,h}\in\MMbar$ that separates $\parenth{\sid, 0}$ from the set $\contactopsetlowner{\loglego{f}}{\loglego{g}}$. This, by \Href{Lemma}{lem:separation_Lowner_problem}, will yield the assertion.

Define $H = A - \id_d.$ 
 We begin by showing that 
\begin{equation}\label{eq:lowner_suff_integral_separation}
\iprod{\parenth{H \oplus \gamma, h}}{\parenth{\sid, 0}} > 0.
\end{equation}
For any  $\delta > 0$,  the function $\tilde{g}$ defined by 
\[
\tilde{g}(x) = e^{-\gamma + \delta} g\parenth{Ax + h}
\]
satisfies the relation $f \leq \tilde{g}$. Moreover, for a sufficiently small $\delta$, we have also
$\int_{\Red} \tilde{g}^s  < \int_{\Red} g^s$.

Define
\[
\tilde{g}_t (x) =  e^{-t(\gamma - \delta)}
g\parenth{\parenth{\id_d + t H} x + th} \in 
\funppos{g}.
\] 		
By \Href{Lemma}{lem:outer-interpolation}, 
$f \leq \tilde{g}_t$ for all $t \in [0,1],$ 	and
\[
				\int_{\Red} {\tilde{g}_t^s} \leq
					\parenth{ \int_{\Red} {g}^s}^{1-t}
					 \parenth{\int_{\Red} \tilde{g}_1^s}^{t}
					\leq \int_{\Red} {g^s}.
\]
Using implication \eqref{item:int_lowner_gammaadmissible} $\Rightarrow$ \eqref{item:int_lowner_admissible_linearized_in_primalweak}    of \Href{Lemma}{lem:lowner_separation_integral}, we have
\[
 0 \leq \iprod{\parenth{H  \oplus (\gamma - \delta), z}}{\parenth{\sid, 0}}  = 
\iprod{\parenth{H  \oplus \gamma, h}}{\parenth{\sid, 0}} - \delta s. 
\]
We conclude that 
\begin{equation*}
\iprod{\parenth{H \oplus \gamma, h}}{\parenth{\sid, 0}} \geq \delta s > 0,
\end{equation*}
and \eqref{eq:lowner_suff_integral_separation} follows.

For $t \in [0,1),$ define 
\[
{g}_t (x) =  e^{t\gamma} g\parenth{\parenth{\id_d + t H}x + th}.
\] 		
That is, $g_0 = g,$  and ${g}_t  \in 
\funppos{g}$ for all $t \in [0,1).$
By the choice of $g_1$ and by \Href{Lemma}{lem:outer-interpolation}, one sees that $f \leq {g}_t$ for all $t \in [0,1].$
The implication \eqref{item:lowner_admissible} $\Rightarrow$  \eqref{item:lowner_admissible_linearized_in_primalweak} in  \Href{Theorem}{thm:admissiblelinearized_lowner} implies that
$\iprod{\parenth{H \oplus \gamma,h}}
{\contactoplowner{\upthing{u}+c}{\upthing{v}}} \leq 0
$
for all $ \parenth{\upthing{u}, \upthing{v}} \in \contactset{\loglego{f}}{\loglego{g}}$.
Combining this with inequality \eqref{eq:lowner_suff_integral_separation},
\Href{Lemma}{lem:separation_Lowner_problem} yields assertion \eqref{item:glmp-lowner-implies-global-minimum} of \Href{Theorem}{thm:lowner_condition_general}, completing the proof of \Href{Theorem}{thm:lowner_condition_general}.

\section{Existence and uniqueness of solutions}\label{sec:existence_uniqueness}

For any  proper log-concave function $f \colon \Red\to[0,\infty),$
there exists a positive constant $C$ such that the 
integral of the proper log-concave function $f$ over any line is
at most $C$. Indeed, it follows, from the existence of constants
 $\Theta, \nu > 0$ depending only on $f$ that satisfy 
\begin{equation}\label{eq:proper_log-concave_bound}
	f(x) \leq \Theta e^{-\nu \enorm{x}} 
\end{equation} 
{ for all } $x\in\Red,$
see \cite[Lemma~2.2.1] 
{brazitikos2014geometry}.
We will use the notation
\[
C_f=\sup_{\ell} \int_{\ell} f,
\]
where the supremum is taken over all lines $\ell$ in $\Red$.

The following technical fact, essentially a rephrasing of Lemma~3.2 of 
\cite{ivanov2022functional}, will allow us to use compactness arguments in finding the optima in the John and L\"owner $s$-problems.

\begin{lem}[Boundedness of the admissible set]\label{lem:boundedness}
For any  proper log-concave function $f \colon \Red\to[0,\infty)$ and any $\delta>0$, 
there exist  $\vartheta,\rho,\rho_1>0$ with the following property. If  
for
a  proper  log-concave function 
$g \colon \Red \to [0,\infty)$ with $g \leq g(0) =1$  and 
$(A\oplus\alpha,a)\in\MM$, the function
$w \colon \Red \to [0,\infty)$ given by 
\[
w(x) = \alpha g\parenth{A^{-1}(x-a)}
\]
satisfies $w \leq f$  and 
$ \int_{\Red} w  \geq\delta$, then the following inequalities hold:
\begin{equation}\label{eq:john_alpha_a_universal_bound}
 \vartheta\leq\alpha \leq \norm{f}
\quad \text{and} \quad 
 \enorm{a}  \leq \rho,
\end{equation}
and
\begin{equation}
\label{eq:john_comparison_operator}
 \frac{\rho_1}{ \int_{\Red} {g}} 
\parenth{\frac{C_{f}}{ C_{g}}}^{1-d} \leq
 \frac{1}{\norm{A^{-1}}} \leq \norm{A} \leq 
\frac{C_{f}}{\vartheta C_{g}}.
\end{equation}
where $C_{g}$ is the maximum of the integral of the restriction  of 
$g$ to a line in $\Red$.
\end{lem}
\begin{proof}
The proof is a minor modification of the proof of a particular case of Lemma~3.2 of \cite{ivanov2022functional}.  
Obviously,  $\alpha \leq \norm{f}.$
To bound $\alpha$ from below, we fix $\vartheta$ with
$\alpha \leq \vartheta.$ Then 
$w \leq \vartheta$, and thus,
\[
 \int_{\Red} w \leq
\int_{\Red} \min\{f(x), \vartheta\} \di x.
\]
Since $f$ is a non-negative function of finite integral, the last expression is 
less than $\delta$ if $\vartheta$ is sufficiently small. Thus, the leftmost 
inequality in  \eqref{eq:john_alpha_a_universal_bound} holds. 
Since $w(a)  = \alpha,$ we conclude that
$a \in [f \geq \vartheta]$ completing the proof of    
\eqref{eq:john_alpha_a_universal_bound}.

We proceed with inequality~\eqref{eq:john_comparison_operator}.
Clearly, $C_f \geq C_w = \alpha \norm{A} C_{g}\geq \vartheta \norm{A} C_{g}$.
 Thus, the rightmost relation in 
\eqref{eq:john_comparison_operator} holds.

By the assumption, we have
\[
  \delta \leq \int_{\Red} w  = \alpha  \abs{\det A}  \cdot
 \int_{\Red} g .
\]
Let $\beta$ be the smallest singular value of $A$.
By the previous inequality and 
since $\alpha \in [\vartheta, \norm{f}],$ we have
\[
0 < \frac{\delta}{\norm{f}} \frac{1}{ \int_{\Red} g} \leq 
\abs{\det A} \leq \beta \cdot \norm{A}^{d-1}.
\]
By the rightmost relation in \eqref{eq:john_comparison_operator}, the existence of 
$\rho_1$ follows.
\end{proof}

\subsection{Existence and uniqueness in the (Positive) John \texorpdfstring{$s$}{s}-problem}

\begin{prop}\label{prop:existence_uniqueness_john}
Let $f, g \colon \Red\to[0,\infty)$ be two proper log-concave functions such that 
$g \leq f.$
Then  John $s$-problem \eqref{eq:john_problem_intro}
and Positive John $s$-problem \eqref{eq:john_problem_pos} have solutions.
Moreover, if $g$ is of bounded support, then the solution to  {Positive John $s$-problem} \eqref{eq:john_problem_pos} is unique.
\end{prop}

We note that  as was shown in \cite{ivanov2022functional}, the solution 
to the {Positive John  $s$-problem} is not necessarily unique without assumption on the boundedness of the support. For example, it is not necessarily unique for the standard  Gaussian density. 

\begin{proof}[Proof of \Href{Proposition}{prop:existence_uniqueness_john}]
The existence of the solutions follows from \Href{Lemma}{lem:boundedness} and a routine compactness argument. 
So, we only need to show the uniqueness of the solution  to the {Positive John $s$-problem} \eqref{eq:john_problem_pos}.
Let $A_1$ and $A_2$ be rank $d$ positive definite matrices, $a_1, a_2 \in \Red$, and
$\alpha_1, \alpha_2>0$ be such that the functions 
\[
h_1(x) = \alpha_1 {g}\! \parenth{A_{1}^{-1}(x - a_1)}
\quad \text{and} \quad 
h_2(x) = \alpha_2 {g}\! \parenth{A_{2}^{-1}(x - a_2)}
\]
are the solutions to Positive John $s$-problem \eqref{eq:john_problem_pos}.
In particular, the integrals of the $s$ power of these functions are equal. 
By \Href{Lemma}{lem:inner-function-interpolation}, $A_1 = A_2$.
Hence, $\alpha_1 = \alpha_2$ as well. 
That is, the liftings of $h_1$ and $h_2$ are translates of each other. The log-concavity of $f$ implies that the set $\lifting{h_1} + [0,2w]$ with non-zero $w \in \Red$ is contained in 
$\lifting{f}$.  
We claim that that there is a position $h$ of ${g}$ below $f$ such that $\int h^s> \int h_1^s$. 
Indeed, consider the function $h_{1.5}(x) = h_1 (x  - w)$. Clearly, 
$\lifting{h_{1.5}} \subset \lifting{h_1} + [0,2w] \subset \lifting{f}$.
Let $h_{1.5}$ attain its maximum at $z$. It means that $z$ belongs to all non-empty level-sets of $h_{1.5}$. These level sets are compact convex sets, since $h_{1.5}$ is a log-concave function of compact support.  Let $S_{\epsilon}$ be the linear transformation that scales $\Red$ in the direction of $w$ by the factor $1 + \epsilon$. Then, for a sufficiently small positive $\epsilon$,  the inclusion 
\[
S_{\epsilon} \parenth{\left[ h \geq \Theta\right] - z} \subset \left[ h_1 \geq \Theta\right] + [0, 2w]
\] 
holds for all $\Theta \in (0, \norm{h}]$.
That is,  $S_{\epsilon} \parenth{\lifting{h_{1.5}} - z} \subset \lifting{f}$.
However, the set on the left-hand side of the last inclusion is the lifting of some position $h$ of ${g}$, completing the proof of \Href{Proposition}{prop:existence_uniqueness_john}.
\end{proof}

\subsection{Existence  in the (Positive) L\"owner \texorpdfstring{$s$}{s}-problem}

By a routine limiting argument, \Href{Lemma}{lem:boundedness} yields the following.
\begin{prop}\label{prop:existence_lowner}
Let $f, g \colon \Red\to[0,\infty)$ be proper log-concave functions such that 
$f \leq g$. Then there are solutions to  L\"owner $s$-problem \eqref{eq:lowner_problem_intro} and to  {Positive L\"owner $s$-problem} \eqref{eq:lowner_problem_pos}.
\end{prop}

We note that as was shown in \cite{ivanov2021functional}, the solution 
to {Positive L\"owner $s$-problem} is not necessary unique even for a radially symmetric function $g$.

\section{The normal cone and the subdifferential}\label{sec:normalcone_subdifferential}

This section contains properties of the normal cone of log-concave functions, that will be needed in the next section, where we discuss when the rather technical conditions of Theorems~\ref{thm:john_condition_general} and \ref{thm:lowner_condition_general} hold.

\subsection{Subdifferential}
We recall several definitions and properties about the subddiferential. 

A vector $p$ is said to be a \emph{subgradient} of a  function 
$\psi  \colon \R^d \to \R \cup \{+ \infty\}$ at the point $x$ if
\begin{equation*}
\psi(y) \geq \psi(x) + \iprod{p}{y-x}
\end{equation*}
for all $y \in \R^d$.
The set of all subgradients of $\psi$ at $x$ is called the \emph{subdifferential}
of $\psi$ at $x$ and is denoted by $\partial \psi(x)$. By definition, we have
\begin{lem}[Subdifferential and normals of the epigraph]\label{lem:subdif_via_ncone}
Let $\psi$ be a convex function on $\R^d.$ Then $p \in \partial \psi(x)$ if and only if
$(p, -1) \in \nfcone{\epi \psi}{(x, \psi(x))}.$
\end{lem}
 The following lemma is a basic property of the subdifferential (see \cite[Theorem 23.5]{rockafellar1970convex}) relating it to the Legendre dual.
\begin{lem}[Subdifferential and Legendre dual]\label{lem:subdif_basics}
 Let $\psi: \R^d \to \R \cup \{+\infty\}$ be a lower semi-continuous convex function. 
 Then the following three conditions on vectors $u, p \in \R^d$ are equivalent
 \begin{enumerate}
\item  $p   \in \partial \psi(u);$
\item  $u \in \partial\slogleg[] \psi(p);$
\item  
$ 
\psi(u) + \slogleg[]\psi(p) = \iprod{p}{u}.
$
\end{enumerate}   
\end{lem}

Also, we will use the following well-known fact.
\begin{lem}\label{lem:subdif_lipchitz}
 Let $\psi: \R^d \to \R \cup \{+\infty\}$ be a lower semi-continuous convex function containing $\delta \ball{d}$ in  the interior of its effective domain. 
 Then there is a constant $L$ such that $\enorm{p} \leq L$ for all 
 $p \in \partial \psi(u)$ and $u \in \delta \ball{d}.$
\end{lem}
 
 \subsection{Explicit formula for contact pairs}
 \label{subsec:normal_cone_via_subdiff}
 
 As an immediate  consequence of \Href{Lemma}{lem:nfcone}, we obtain
\begin{lem}[Horizontal normals of $\lifting{f}$]\label{lem:ncone_flat_normal}
Let $f \colon \Red \to [0, +\infty)$ be an upper semi-continuous log-concave function containing the origin in the interior of its support. 
Let $\upthing{u} = (u, f(u)) \in \essgraph{f}$ and $\upthing{v} = (v, 0) \in \nfcone{\lifting{f}}{\upthing{u}}$ with $\iprod{\upthing{u}}{\upthing{v}} = 1.$ 
Then   
\[
 \upthing{v} =  \frac{(p,0)}{\iprod{p}{u}} 
    \quad \text{and} \quad  \iprod{p}{u} > 0
 \]
 for some non-zero $p \in \nfcone{\supp f}{u}.$
\end{lem} 
Lemmas~\ref{lem:nfcone} and \ref{lem:subdif_via_ncone},
yield the following.
\begin{lem}[Non-horizontal normals of $\lifting{f}$]
\label{lem:v_via_grad}
Let $\psi \colon \Red \to \R \cup \{+\infty\}$ be a convex function  containing $u$ 
in  its domain.  Set $f = e^{-\psi}$ and $\upthing{u} = (u, f(u))$. Let 
$\upthing{v}= (v, \nu)$ with $\nu \neq 0.$
The following assertion are equivalent:
\begin{enumerate}
\item $\upthing{v} \in \nfcone{\lifting{f}}{\upthing{u}}$  and
 $\iprod{\upthing{u}}{\upthing{v}} =1 $
\item   \[
 \upthing{v} =  
  \frac{\parenth{p, \frac{1}{f(u)} }}{1 + \iprod{p}{u}} 
  \quad \text{and} \quad 1 + \iprod{p}{u} > 0
 \]
for some  $p \in \partial \psi(u).$ 
\end{enumerate}    
\end{lem}

We conclude the following.

\begin{lem}\label{lem:starlike_criteria}
Let $g\colon \Red \to [0, + \infty)$ be a proper log-concave function containing the origin in the interior of support. Then, a set $U \subset \Red$ is a \wlocstar$g$ if and only if the inequality $\iprod{p}{u} > -1$ holds for all $u \in U \cap \supp g$ and
 $p \in  \partial(-\ln g)(u).$
\end{lem}
\begin{proof}
Fix $u \in U,$ denote $\upthing{u} = (u, g(u)),$ and let $\upthing{v}= \parenth{v, \frac{\nu}{g(u)}} \in \nfcone{\lifting{g}}{\upthing{u}}.$ If $\nu \neq 0.$
Lemmas~\ref{lem:nfcone} and \ref{lem:subdif_via_ncone},
yield  that $\nu  > 0$ and  
$v = \nu p$ for some  $p \in \partial \psi(u).$ 
Hence, $\iprod{\upthing{v}}{\upthing{u}} = \nu (1 + \iprod{p}{u}).$
If $\nu = 0,$ then $u$ is a point of the boundary of $\supp g.$ By 
\Href{Lemma}{lem:nfcone} and since the origin is in the interior of the convex set $\supp g,$ the  inequality $\iprod{\upthing{v}}{\upthing{u}} = \iprod{v}{u} > 0$ holds.
\end{proof}

\begin{remark}
Let $g = e^{-\psi} \colon \Red \to [0, + \infty)$ be an upper semi-continuous function, 
and let $u$ be a point on the boundary of $\supp g$ at which $\psi$ is not sub-differentiable.  It is not hard to show that $\nfcone{\lifting{g}}{(u, g(u))} = \nfcone{\supp g }{u}$ holds in this case. 
\end{remark}

\section{Properties of the set of contact pairs}\label{sec:boundedcontactpairs}

The main topic of the section is the question of when the set of contact pairs 
$\contactset{f}{g}$ is bounded in $\Redp \times \Redp$, as this is a crucial 
condition in Theorems~\ref{thm:john_condition_general} and 
\ref{thm:lowner_condition_general}.

Assume that $g \leq f$. Then $\contactset{f}{g} \subseteq \contactset{g}{g}$ and hence,
it suffices to impose conditions on $g$ to guarantee the boundedness of $\contactset{f}{g}$. 

\subsection{A difficulty: flat zeros}

In order to explain the difficulty of guaranteeing that the contact set is bounded, we first return to the setting of convex sets. 
Assume that $K$ is a compact convex subset of $\Red$ containing the origin in the interior. 
Then  for any $u \in \bd{K}$, the set 
\[
N_u =\{v \st \iprod{u}{v} = 1 \quad \text{and} \quad v \in \nfcone{K}{u}\}
\]
 is a closed subset of the boundary of ${K}^{\polar}$.
 Hence, the set of contact pairs
  \[
\mathcal{C} = \{(u, v) \st u \in \bd{K}, \; v \in N_u\} 
\]
is a closed subset of the compact set $\bd{K}\cap\bd{{K}^{\polar}}$, and hence, is compact. 

Next, consider the epigraph of a convex function $\psi$ with bounded domain.
It will be an unbounded convex set in $\Redp$. So, when we turn to the lifting of 
the corresponding log-concave function $g=e^{-\psi}$, we may find ill-behaved points. Namely, at any 
point $u \in \cl{\dom \psi} \setminus 
\dom \psi$, the normal cone of the epigraph of $\psi$ is undefined at $(u, 
\psi(u))$, since there is no such point, as $\psi(u)=\infty$. However, the 
normal cone to the lifting of $g$ at $(u, g(u)) = (u, 0)$ 
is well defined and non-empty by \Href{Corollary}{cor:normalcone_nonempty}. Most importantly, $\nfcone{\lifting{g}}{(u, 0)}$ 
may contain $e_{d+1}$. It is this particular case that requires additional care.  

\begin{dfn}\label{def:flatzero}
 For a log-concave function $g\colon\Red\to\Re$, we call a point $u\in\cl{\supp 
g}$ a \emph{flat zero}, if $g(u)=0$ and $e_{d+1}\in\nfcone{\lifting{g}}{(u, 0)}$.
\end{dfn}

To see why flat zeros pose a difficulty, assume that $g$ and $u$ are as in the definition above, and $f$ is a function with $g\leq f$ such that $f=g$ on some neighborhood of $u$. Let $v\in\Red$ be an outer normal vector of the support hyperplane of $\supp f$ at $u$ with $\iprod{u}{v}=1$. 
Then $\big((u,0),(v,\nu)\big)$ is in $\contactsetnr{f}{g}$ for all $\nu\in\Re$, and hence, $\contactset{f}{g}$ is not bounded.


Consider the following examples, $g_1,g_2 \colon \Red \to [0, +\infty)$.
\[
g_1(u) = 
 \begin{cases}
\parenth{1 -  \enorm{u}^2}^2,& 
\text{ if } \enorm{u} < 1\\
0,&\text{ otherwise}.
 \end{cases}
\]
\[
g_2(u) = 
 \begin{cases}
{e^{-\frac{1}{1 - \enorm{u}}}},& 
\text{ if } \enorm{u} < 1\\
0,&\text{ otherwise}.
\end{cases}
\]

Clearly, both $g_1$ and $g_2$ are proper log-concave functions of bounded support.
It is easy to see that the normal cone of both of them is 
$\nfcone{\lifting{g_i}}{(u,0)} =\{(\alpha u,\nu) \st \alpha \geq 0, \nu \in \R \}$ for any unit vector
$u \in \Red$. That is, the sets $\contactsetnr{g_1}{g_1}$ and 
$\contactsetnr{g_2}{g_2}$ are unbounded, and every unit vector is a flat zero. Even though for any fixed $\upthing{u}=(u,0)$ with $\enorm{u}=1$, the set $\{(\upthing{u},\upthing{v})\in\contactset{g_i}{g_i}\}$ is bounded, since only horizontal vectors $\upthing{v}$ are present in it, but if $\enorm{u}<1$ is close to 1, then $\{(\upthing{u},\upthing{v})\in\contactset{g_i}{g_i}\}$ becomes arbitrarily large. Hence $\contactset{g_i}{g_i}$ is not bounded.

Yet, there is a major difference between these two functions. The function $g_1$ was studied in \cite{ivanov2022functional} and the conditions of optimality equivalent to that of \Href{Theorem}{thm:john_condition_general} were obtained therein for an arbitrary proper log-concave function $f$. It was possible, since $g_1$ is $q$-concave with $q=1/4$ (see page~\pageref{page:qconcave} for the definition), and hence -- as we will see --, its flat zeros can be removed by taking the $q$-th power. On the other hand, $g_2$ is not $q$-concave for any $q>0$.

\subsection{Sufficient conditions}\label{subsec:sufficient_conditions}

In this subsection, we show that if $g(0)$ is close to $\min g$, then the conditions in our main theorems (\Href{Theorems}{thm:john_condition_general} and \ref{thm:lowner_condition_general}) hold essentially with any $f$.

\begin{lem}\label{lem:sufficient_condition_locstar_and_bounded}
Let $g = e^{-\psi} \colon \Red  \to [0, +\infty)$ be a function satisfying our Basic Assumptions (see page~\pageref{assumptions:basic}). 
Denote the minimum of $\psi$ by $m$. Assume that  the inequality $\psi(0) < m+1$ holds.
Then $\supp g$ is a \wlocstar$g$. Additionally, let $U$ be a  subset of $\Red$  such that
$\inf\limits_{x \in \ \!\! U \ \! \cap \ \! \supp g } g(x) > 0,$ then 
then the 
set
$
C = \big\{
((u, f(u)), \upthing{v}) \in \contactset{g}{g} \st u \in U 
\big\}
$ is bounded. 
\end{lem}
\begin{proof}
 Fix $\upthing{u}=(u,g(u)) \in \essgraph g$ and let 
$\upthing{v} = (v,\nu) \in \nfcone{\lifting{g}}{(u,g(u))} \setminus \{0\}.$ 

Consider the case when $\nu=0$. By \Href{Lemma}{lem:nfcone}, $u$ belongs to the boundary of $\supp{g}$ and $v \in \nfcone{\supp{g}}{u} \subset \Red$. 
Since  $\supp g$ is a convex set containing the origin in its interior, 
there are positive constants $\delta_0$ and $\delta_1$  independent of $u$ and $\upthing{v}$ such that $\enorm{\upthing{v}} < \delta_0$ and  $\iprod{\upthing{v}}{\upthing{u}} > \delta_1$.

Now, assume that $\nu \neq 0$.
By \Href{Lemma}{lem:nfcone} and the convexity of $\psi$, we have that $\nu > 0$ and 
\[
 \iprod{v}{u}  \geq {\nu}{g(u)}\parenth{ \psi(u) - \psi (0)} \geq
 {\nu}{g(u)} \parenth{ m - \psi (0)}.
\]
Hence,
\begin{equation}\label{eq:admissible_center}
\iprod{\upthing{v}}{\upthing{u}} =  \iprod{v}{u} + \nu g(u) \geq
{\nu}{g(u)} \parenth{ m - \psi (0) + 1} > 0.
\end{equation}
We conclude that $\supp g$ is a \wlocstar$g$.

Set
$L= \inf\limits_{x \in U \cap \supp g} g(x).$ 
Now, assume that $\iprod{\upthing{v}}{\upthing{u}} = 1.$
Then \eqref{eq:admissible_center} yields
\[
\nu  \leq \frac{1}{\parenth{ m - \psi (0) + 1}L }
\]
for any $( (u, f(u)), \upthing{v}) \in \contactset{g}{g}$ such  that $u \in U$. 
Thus, $\nu$ cannot be too large.
\Href{Lemma}{lem:subdif_lipchitz}  and \Href{Lemma}{lem:nfcone} imply that the set 
of contact pairs
$
((u,f(u)), \upthing{v}) \in \contactset{g}{g},
$
 where $u$ is in a sufficiently small neighborhood $V$  of the origin, is bounded. 
 Using convexity again and  identity \eqref{eq:normalconelevelset} of \Href{Lemma}{lem:nfcone}, we see that  
 $\enorm{v}$ cannot be too large outside $V,$ and thus, the proof of the lemma is complete.

\end{proof}

As a direct consequence of this result, we obtain the following.
\begin{cor}\label{cor:suffiient_cond_starlike+bounded}
Let $\gbase \colon \Red \to [0, \infty)$ be an upper semi-continuous function such that 
$\logconc{\gbase}$ satisfies  our Basic Assumptions (see page~\pageref{assumptions:basic}). 
Denote the minimum of $-\ln \logconc{\gbase}$  by $m.$ Assume  
the inequality 
\[
-\ln \logconc{\gbase}(0) < m +1 
\]
holds. 
Then for any proper log-concave function $f \colon \Red \to [0, +\infty),$ 
the set $\contactpoint{f}{\gbase}$ is a \wlocstar$f$.
Moreover, if the infimum of $\gbase$ taken over the set $\contactpoint{f}{\gbase} \cap \supp \gbase$ is strictly greater than zero, then the set
$\contactset{f}{\gbase}$ is bounded. 
\end{cor}
\begin{remark}
If $\logconc{\gbase}$ is of compact support in the assertion of \Href{Corollary}{cor:suffiient_cond_starlike+bounded}, then for any function $f$ such that the set
$S = \contactpoint{f}{\gbase} \cap \supp \gbase$ is non-empty, the infimum of $\gbase$ taken over $S$ is strictly greater than zero.
\end{remark}

\subsection{Taking the reduced set of contact pair yields no loss}

We recall that in the assertions of our theorems formulated in the Introduction,
we consider the full set of contact pairs, $\contactsetnr{f}{g}$, and not the \emph{reduced} set $\contactset{f}{g}$. The following simple observation allows us to do so.

\begin{lem}[The vertical component of $\upthing{v}$ may be ignored at a zero]\label{lem:flattened_normal_at_zero}
Let $f \colon \Red \to [0, +\infty)$ be an upper semi-continuous log-concave 
function, and $u \in \cl{\supp f}$ be such that $f(u) = 0$. Let $\upthing{v} 
= (v, \nu) \in \nfcone{\lifting{f}}{\upthing{u}}$, where $\upthing{u}=(u, 
f(u))$,  be such that $\iprod{\upthing{v}}{\upthing{u}} = 1$.
 Then  $v \in \nfcone{\supp f}{u}$, and
 \[
 \contactopjohn{\upthing{u}}{\upthing{v}} = \contactopjohn{\upthing{u}}{(v,0)} 
 \quad  \text{and} \quad  \contactoplowner{\upthing{u}}{\upthing{v}} = \contactoplowner{\upthing{u}}{(v,0)}.
 \]
\end{lem}

Thus, if no contact pairs from $\contactsetnr{f}{g}$ (resp., $\contactsetnr{\loglego{f}}{\loglego{g}}$) satisfy the equations in 
\Href{Theorem}{thm:john_condition_general} (resp., 
\Href{Theorem}{thm:lowner_condition_general}), then one can consider only horizontal normals at the zero level when studying separation of the set of extended contact operators from the point $\parenth{\id_{d} \oplus s, 0}$ in $\MMbar$.

A more geometric, less algebraic explanation of this fact is that in the proof of   
\Href{Theorem}{thm:john_condition_general} (resp., 
\Href{Theorem}{thm:lowner_condition_general}), a point  $ \upthing{u}=(u, 0)$ remains in 
$\Red$ under the corresponding homotopy. That is, the last coordinate of the 
normal vector does not play a role.

\section{Radially symmetric functions}\label{sec:radially_symmetric}

We call a function $f$ on $\Red$ \emph{radially symmetric},
if it is of the form $f(x) = F(\enorm{x})$, where $F$ is a function on 
$[0,+\infty)$. Clearly, the sets of positions and positive positions  of  a radially symmetric function coincide. Consequently, the identities 
\eqref{eq:functional_glmp} and \eqref{eq:functional_glmp-lowner} are the necessary and sufficient condition in \Href{Theorem}{thm:john_condition_general} and 
\Href{Theorem}{thm:lowner_condition_general}, respectively. This way, the 
results of \cite{ivanov2022functional}  can be recovered from  \Href{Theorem}{thm:john_condition_general}.

\section{The \texorpdfstring{$q$}{q}-concave case}\label{sec:qconcave}

\subsection{Immediate corollaries of Theorems~\ref{thm:john_condition_general} and \ref{thm:lowner_condition_general}}

Fix $s > 0$ and $q > 0$.
If a function $g \colon\Red\to[0,+\infty)$ is a proper $q$-concave function, then
$\lifting{g^q}$ is a compact convex set with non-empty interior in $\Redp$, which yields the following.

\begin{lem}\label{lem:qpowerok}
Let  $g \colon\Red\to[0,+\infty)$ be a proper $q$-concave function containing the origin in the interior of the support. Then $\supp{g^q}$ is a \locstar $g^q$ and 
the set $\contactset{g^q}{g^q}$ is bounded. 
\end{lem}  
Thus, we have the following results.
\begin{cor}\label{cor:john_cond-q-concave}
Fix $s>0$, and let $f,g \colon\Red\to[0,+\infty)$ be two proper log-concave functions.
Additionally, assume  $g$ is $q$-concave with some $q>0$ and contains the origin in the interior of its support.
\begin{enumerate}
		\item
			If $h=g$  is  a local maximizer in  John $s$-problem \eqref{eq:john_problem_intro},
			then there exist contact pairs 
			$(\tilde{u}_1,\tilde{v}_1)$, $\dots$,  
			$(\tilde{u}_m,\tilde{v}_m)$ $\in \contactset{f^q}{g^q}$
			and positive weights $\tilde{c}_1,\dots,\tilde{c}_m$  satisfying
			\begin{equation}\label{eq:functional_glmp_concave}
	\sum_{i=1}^{m} \tilde{c}_i {u}_i \otimes {v}^\prime_i = 
	\id_{d}, \quad 
	\sum_{i=1}^{m} \tilde{c}_i f^q(u_i){\nu}^{\prime}_i =  \frac{s}{q}
		\quad\text{and}\quad
		\sum_{i=1}^{m} \tilde{c}_i {v}^\prime_i=0,
\end{equation}
where $\tilde{u}_i=(u_i, f^q(u_i))$ and $\tilde{v}_i=(v^\prime_i,\nu^\prime_i)$.
		\item
			If there exist contact pairs and positive weights  satisfying  equation \eqref{eq:functional_glmp_concave},
			then $g$ is a global maximizer in  Positive position John $s$-problem \eqref{eq:john_problem_pos}.
	\end{enumerate}
\end{cor}
\begin{proof}
We consider the corresponding John $\frac{s}{q}$-problem for functions $f^q$ and $g^q$, and by \Href{Lemma}{lem:qpowerok} we may apply \Href{Theorem}{thm:john_condition_general}.
\end{proof}
\begin{cor}\label{cor:lowner_cond-q-concave}
Fix $s>0$, and let $f,g \colon\Red\to[0,+\infty)$ be two proper log-concave functions.
Assume that $\loglego{g}$ is $q$-concave with some $q>0$, and contains the origin in the interior of its support.
Then the following hold.
	\begin{enumerate}
		\item
			If $h=g$ is a local minimizer in  L\"owner $s$-problem \eqref{eq:lowner_problem_intro} for $f$ and $g$,
			then there exist contact pairs 
			$(\tilde{u}_1,\tilde{v}_1), \dots,  
			(\tilde{u}_m,\tilde{v}_m) \in 
			\contactset{\loglego{\parenth{g^q}}}{\loglego{\parenth{f^q}}}$
			and positive weights $\tilde{c}_1,\ldots,\tilde{c}_m$ such that
		\begin{equation}\label{eq:functional-glmp-lowner_concave}
		\sum_{i=1}^{m} \tilde{c}_i {v}^\prime_i \otimes {u}^\prime_i = 
		 \id_{d}, \quad 
		 \sum_{i=1}^{m} \tilde{c}_i \loglego{\parenth{g^q}}(u_i^\prime) \cdot \nu_i^\prime =  \frac{s}{q}.
					\quad\text{and}\quad
					\sum_{i=1}^{m} \tilde{c}_i\loglego{\parenth{g^q}}(u_i^\prime) \cdot \nu_i^\prime u_i^\prime=0,
        \end{equation} 
        where $\tilde{u}_i=(u_i^\prime, \loglego{\parenth{g^q}}(u_i^\prime))$ and $\tilde{v}_i=(v_i^\prime,\nu_i^\prime)$.
		\item
			If there exist contact pairs and positive weights satisfying equation \eqref{eq:functional_glmp-lowner},
			then $h=g$ is a global maximizer in Positive position L\"owner $s$-problem \eqref{eq:lowner_problem_pos} for $f$ and $g$.
	\end{enumerate}
\end{cor}
\begin{proof}
Observe that 
\[
\loglego{\parenth{g^{q}}}(y) = \parenth{\loglego{g}\!\parenth{\frac{y}{q}}}^{q}
\]
for any  $q > 0.$ Hence, $\loglego{\parenth{g^q}}$ is $1$-concave.
 We consider the corresponding L\"owner $\frac{s}{q}$-problem for functions $f^q$ and $g^q$, and by \Href{Lemma}{lem:qpowerok} we may apply \Href{Theorem}{thm:lowner_condition_general}.
\end{proof}

We note that in these two corollaries, we choose the origin arbitrarily inside the interior of the support of corresponding functions. 
To obtain the theorems formulated in the Introduction, we need a more subtle argument.

\subsection{Taking power and algebraic identities involving the contact pairs}

In this subsection, we discuss how replacing $f$ and $g$ by $f^q$ and $g^q$ changes the form of identities \eqref{eq:functional_glmp} and \eqref{eq:functional_glmp-lowner}.

\begin{lem}\label{lem:equiv_john_cond_power}
Fix $s > 0$ and $ q \in ( 0, 1]$. Let $f\colon\Red \to [0,+\infty)$ be a proper log-concave function,
and let $\gbase \colon \Red \to [0, \infty)$ be an upper semi-continuous function 
such that 
\begin{itemize}
\item $\logconc{\gbase}$ satisfies our Basic Assumptions (see page~\pageref{assumptions:basic}); 
\item the set of contact points $\contactpoint{f}{\gbase}$ is a \wlocstar$f$.
\end{itemize}
Set $\upthing{u}_i = (u_i, f(u_i))$ and $\tilde{u}_i = (u_i, f^q(u_i)).$
Then the following assertions are equivalent:
\begin{itemize}
\item  There are contact pairs $(\upthing{u}_1,\upthing{v}_1), \dots,  
			(\upthing{u}_m,\upthing{v}_m) \in \contactset{f}{\gbase}$
			and positive weights $c_1,\ldots,c_m$ that satisfy the identities in 
			\eqref{eq:functional_glmp}.
\item  There are contact pairs $(\tilde{u}_1,\tilde{v}_1), \dots,  
			(\tilde{u}_m, \tilde{v}_m) \in \contactset{f^q}{\gbase^q}$
			and positive weights $\tilde{c}_1,\ldots,\tilde{c}_m$  that satisfy the identities in 
			\eqref{eq:functional_glmp_concave}.
\end{itemize} 
\end{lem}

\begin{proof}
Fix $\upthing{u}_i = (u_i, f(u_i)) \in \essgraph{f}$ and
 $\upthing{v}_i = (v_i, \nu_i) \in \nfcone{\lifting{f}}{\upthing{u}_i}$  with 
 $\iprod{\upthing{u}_i}{\upthing{v}_i} = 1.$

We consider two cases.

\emph{Case of a horizontal normal:} 
Assume that $\nu_i =0$.  
By Lemmas~\ref{lem:nfcone_at_zero} and \ref{lem:ncone_flat_normal}, the following assertion are equivalent.
\begin{itemize}
\item $\parenth{\upthing{u}_i, \upthing{v}_i} \in \contactset{f}{\gbase};$
\item $\parenth{\tilde{u}_i, \tilde{v}_i} \in \contactset{f^q}{\gbase^q},$ 
where $\tilde{u}_i =(u_i, f^q(u_i))$ and $\tilde{v}_i = \upthing{v}_i = (v_i, 0)$.
\end{itemize} 
Finally, 
\[ \contactopjohn{\upthing{u}_i}{\upthing{v}_i} =
\parenth{u_i \otimes v_i\oplus  f(u_i) \nu_i, v_i } = 
 \parenth{u_i \otimes v_i\oplus  0, v_i }
 = \contactopjohn{\tilde{u}_i}{\tilde{v}_i} \]
We set $\tilde{c}_i = c_i$ in this case.

\emph{Case of a non-horizontal normal:} 
Assume that $\nu_i > 0$. By the definition of $\contactset{f}{\gbase}$, we have $f(u_i) > 0$. 

Since  $\contactpoint{f}{\gbase}$ is a \wlocstar${f}$ 
and by \Href{Lemma}{lem:v_via_grad},
\[
 \upthing{v}_i  =
  \frac{\parenth{p_i, \frac{1}{f(u_i)} }}{1 + \iprod{p_i}{u_i}} 
\]
and $1 + \iprod{p_i}{u_i} > 0 $ for some $p \in \partial (- \ln f)(u_i)$.
Hence, $1 + \iprod{p_i}{u_i} > 0$ and  $1 + q\iprod{p_i}{u_i} > 0.$  
Thus, 
\Href{Lemma}{lem:v_via_grad} implies that the following assertion are equivalent:
 \begin{itemize}
 \item $\parenth{\upthing{u}_i, \upthing{v}_i} \in \contactset{f}{\gbase}$ and 
$
 \upthing{v}_i =  
  \frac{\parenth{p_i, \frac{1}{f(u_i)} }}{1 + \iprod{p_i}{u_i}}
$
for some $p_i \in \partial(- \ln f)(u_i);$
\item  $\parenth{\tilde{u}_i, \tilde{v}_i} \in \contactset{f^q}{\gbase^q},$ 
where  $
 \tilde{v}_i =  
  \frac{\parenth{q p_i, \frac{1}{f^q(u_i)} }}{1 + q \iprod{ p_i}{u_i}}
$ for $p_i \in \partial(- \ln f)(u_i)$
and 
$
\tilde{u}_i = (u_i, f^q(u_i)).
$
 \end{itemize}
By substitution,
\[ 
c_i \contactopjohn{\upthing{u}_i}{\upthing{v}_i} =  
c_i \frac{\parenth{({u}_i \otimes {p}_i) \oplus 1, p_i}}{1 + \iprod{p_i}{u_i}} 
\quad \text{and} \quad
 \tilde{c}_i \contactopjohn{\tilde{u}_i}{\tilde{v}_i} = 
 c_i \frac{\parenth{({u}_i \otimes {p}_i) \oplus \frac{1}{q}, p_i}}{1 + \iprod{p_i}{u_i}} ,
 \]
 where $\tilde{c}_i = \frac{c_i}{q} \frac{1 + q\iprod{p_i}{u_i}}{1 + \iprod{p_i}{u_i}}$.
\end{proof}

\begin{lem}\label{lem:equiv_lowner_cond_power}
Fix $s > 0$ and $ q \in ( 0, 1]$. Let $f,g \colon\Red \to [0,+\infty)$ be be two proper log-concave functions such that
\begin{itemize}
\item $\loglego{g}$ satisfies our Basic Assumptions (see page~\pageref{assumptions:basic}); 
\item the set of contact points $\contactpoint{\loglego{f}}{\loglego{g}}$ is a \wlocstar$\loglego{f}$.
\end{itemize}
Set $\upthing{u}_i = (u_i, \loglego{f}(u_i))$ and 
$\tilde{u}_i = \parenth{q{u_i}, \loglego{(f^q)} \!\parenth{{qu_i}}}.$
Then the following assertions are equivalent:
\begin{itemize}
\item  There are contact pairs $(\upthing{u}_1,\upthing{v}_1), \dots,  
			(\upthing{u}_m,\upthing{v}_m) \in \contactset{\loglego{f}}{\loglego{g}}$
			and positive weights $c_1,\ldots,c_m$ that satisfy the identities in 
			\eqref{eq:functional_glmp-lowner}.
\item  There are contact pairs $(\tilde{u}_1,\tilde{v}_1), \dots,  
			(\tilde{u}_m, \tilde{v}_m) \in \contactset{\loglego{(f^q)}}{\loglego{(g^q)}}$
			and positive weights $\tilde{c}_1,\ldots,\tilde{c}_m$  that satisfy the identities in 
			\eqref{eq:functional-glmp-lowner_concave}.
\end{itemize} 
\end{lem}
\begin{proof}
Observe that 
\[
\loglego{\parenth{f^{q}}}(y) = \parenth{\loglego{f}\!\parenth{\frac{y}{q}}}^{q}
\]
for any  $q > 0.$
The rest of the proof is similar to that of the previous lemma. We omit the details.

The corresponding substitutions are as follows: 
$\tilde{v}_i = \parenth{\frac{v_i}{q}, 0}$  and $\tilde{c}_i = c_i$ in the case 
$\nu_i= 0$, and  
$\tilde{v}_i = \frac{\parenth{p_i, \frac{1}{\parenth{\loglego{f}(q_i)}^q} }}{1 + q \iprod{ p_i}{u_i}}
$ for $p_i \in \partial(- \ln \loglego{f})(u_i)$ and   $
  \tilde{c}_i = 
 \frac{c_i}{q} 
  \frac{1 + q\iprod{p_i}{u_i}}{1 + \iprod{p_i}{u_i}}
$ in the case
 $\nu_i > 0$.
\end{proof}

\subsection{Taking power and the boundedness of contact pairs} 

Fix $q \in (0,1]$. If one considers $f^{q}$ and $g^q$ instead of $f$ and $g$,
then the corresponding liftings are ``more'' concave, which implicitly implies that the corresponding problem is easier to solve. In some sense, this is indeed true. 
\Href{Lemma}{lem:starlike_criteria} ensures that if $\contactpoint{f}{g}$ is \wlocstar$f$,
then $\contactpoint{f^q}{g^q}$ is \wlocstar$f^q$. The following result shows that taking power of the functions does not destroy the boundedness of contact pairs.

\begin{lem}[Replacing $g$ by $g^{\alpha}$ and boundedness]\label{lem:power_of_bounded}
Let $g \colon \Red \to [0, +\infty)$ be a proper log-concave function such that the set 
$\contactset{g}{g}$ is bounded. Then for any $\alpha \in (0,1],$ the function 
$g^{\alpha}$ is a proper log-concave function and  the set 
$\contactset{g^\alpha}{g^\alpha}$ is bounded.
\end{lem}
\begin{proof}
We only need to show that if for all $u\in\cl{\supp g}$, we take all possible normals $\upthing{v}=(v,\nu)$ to $\lifting{g^\alpha}$ at $\upthing{u}=(u,g(u))$, then we obtain a bounded set. We fix a $u$.
In the case when $\nu=0$, by \Href{Lemma}{lem:ncone_flat_normal}, $\lifting{g}$ and $\lifting{g^\alpha}$ have the same set of horizontal normals at $u$. Thus, we may assume that $\nu>0$, and hence $u\in\supp g$.

Using  \Href{Lemma}{lem:v_via_grad}, we see that 
 $v = \frac{p}{1 + \iprod{p}{u}}$ for some non-zero $p \in \partial \psi(u)$ and 
 $1 + \iprod{p}{u} > 0$, where $g=e^{-\psi}$.
Using  \Href{Lemma}{lem:v_via_grad} again, we see that  
$\upthing{v}_\alpha = (v, \nu_\alpha)\in  \nfcone{\lifting{g^\alpha}}{(u, g^\alpha(u))}$ with some $\nu_\alpha > 0$ if and only if
\[
\upthing{v}_\alpha = 
\frac{\parenth{ \alpha p, \frac{1}{g^\alpha(u)} }}{1 + \alpha\iprod{ p}{u}}
\] 
for some $p \in \partial \psi(u).$
Note $1 + \alpha \iprod{p}{u} > \alpha (1 + \iprod{p}{u}) > 0.$ 
Since $g$ is bounded, there is a constant $L>0$ such that
\[
\frac{1}{g^{\alpha}(u)} \leq \frac{L}{g(u)}
\]
for all $u \in \supp{g}.$ 
Thus, 
\[
\enorm{\upthing{v}_\alpha} \leq \parenth{1 + \frac{L}{\alpha}}
\enorm{\frac{\parenth{p, \frac{1}{g(u)}}}{1 + \iprod{p}{u}}}.
\] 
By \Href{Lemma}{lem:v_via_grad}, the latter vector is in $\nfcone{\lifting{g}}{f}$, and the lemma follows.
\end{proof} 

Using  \Href{Lemma}{lem:power_of_bounded} and 
 \Href{Lemma}{lem:starlike_criteria}, we see that if the functions $f,\gbase\colon\Red\to[0,+\infty)$ satisfy the assumptions of \Href{Theorem}{thm:john_condition_general}, then the functions $f^\alpha,\gbase^\alpha \colon \Red\to[0,+\infty)$ satisfy it as well for any $\alpha \in (0,1].$
It means that the John conditions in the corresponding problems are fulfilled simultaneously, which is ensured by \Href{Lemma}{lem:equiv_john_cond_power}.

Similarly, using  \Href{Lemma}{lem:power_of_bounded} and 
 \Href{Lemma}{lem:starlike_criteria}, we see that if the functions $f, g\colon\Red\to[0,+\infty)$ satisfy the assumptions of \Href{Theorem}{thm:lowner_condition_general}, then the functions $f^\alpha,g^\alpha \colon \Red\to[0,+\infty)$  with the same $c$ satisfy it as well for any $\alpha \in (0,1].$
It means that the L\"owner conditions in the corresponding problems are fulfilled simultaneously, which is ensured by \Href{Lemma}{lem:equiv_lowner_cond_power}.

\section{Proofs of the results presented in the Introduction}\label{sec:introproofs}

\begin{proof}[Proof of \Href{Theorem}{thm:john_intro}]
The proof mostly repeats the argument we used for \Href{Theorem}{thm:lowner_intro} in \Href{Section}{sec:radially_symmetric}.
Since $f$ takes only positive values and the support of $g$ is bounded, we see that
\[
\inf\limits_{x \in\ \contactpoint{f}{g}\ \cap\ \supp g} g (x) > 0.
\]
By \Href{Lemma}{lem:sufficient_condition_locstar_and_bounded}, we have that 
$\contactpoint{f}{g}$ is a \wlocstar$f$ and $\contactset{f}{g}$ is bounded. 
Thus, the assumptions on functions in \Href{Theorem}{thm:john_condition_general} are fulfilled. The necessity of  condition \eqref{eq:functional_glmp_intro} follows from this and \Href{Lemma}{lem:v_via_grad}. 
By the argument from \Href{Section}{sec:radially_symmetric} for radially symmetric,
 function the sufficiency of the corresponding conditions follows as well.
\end{proof}

\begin{proof}[Proof of \Href{Theorem}{thm:john_intro-concave}]
If $q > 1,$ then $g$ is $1$-concave. Since $\lifting{g}$ is a compact convex set in $\Redp$, the assumptions on the functions in \Href{Theorem}{thm:john_condition_general} are fulfilled. The necessity of condition \eqref{eq:functional_glmp_concave_intro} follows.

Assume that $q \in (0,1].$ Then
by \Href{Lemma}{lem:sufficient_condition_locstar_and_bounded}, 
$\contactpoint{f}{g}$ is a \wlocstar$f$,  
and $\contactset{f^q}{g^q}$ is bounded, since $\lifting{g^q}$ is a compact convex set in $\Redp$. 
Thus, there are  contact pairs 
			$(\tilde{u}_1,\tilde{v}_1)$, $\dots$,  
			$(\tilde{u}_m,\tilde{v}_m)$ $\in \contactset{f^q}{g^q}$
			and positive weights $\tilde{c}_1,\dots,\tilde{c}_m$  satisfying
		\eqref{eq:functional_glmp_concave}. 
		The necessity of condition \eqref{eq:functional_glmp_concave_intro} follows from \Href{Lemma}{lem:equiv_john_cond_power}.

By the argument from \Href{Section}{sec:radially_symmetric} for radially symmetric,
 function the sufficiency of the corresponding conditions follows as well.
\end{proof}

\begin{proof}[Proof of \Href{Theorem}{thm:lowner_intro}]
Since $\loglego{f}$ takes only positive values and the support of $\loglego{g}$ is bounded, we see that
\[
\inf\limits_{x \in\ \contactpoint{\loglego{f}}{\loglego{g}}\ \cap\ \supp \loglego{g}} \loglego{g} (x) > 0.
\]
By \Href{Lemma}{lem:sufficient_condition_locstar_and_bounded}, we have that 
$\contactpoint{\loglego{f}}{\loglego{g}}$ is a \wlocstar$f$ and $\contactset{\loglego{f}}{\loglego{g}}$ is bounded. 
Thus, the assumptions on functions in \Href{Theorem}{thm:lowner_condition_general} are fulfilled. The necessity of  condition \eqref{eq:functional_glmp_lowner_intro} follows from this and \Href{Lemma}{lem:v_via_grad}. 
By the argument at the beginning of the present section, the sufficiency of the corresponding conditions follows as well.
\end{proof}

\begin{proof}[Proof of \Href{Theorem}{thm:lowner_intro-concave}]
If $q > 1,$ then $\loglego{g}$ is $1$-concave. Since $\lifting{\loglego{g}}$ is a compact convex set in $\Redp$, the assumptions on functions in \Href{Theorem}{thm:lowner_condition_general} are fulfilled.  The necessity of condition \eqref{eq:functional_glmp-lowner-conc_intro} follows.

Assume now $q \in (0,1].$ Then
by \Href{Lemma}{lem:sufficient_condition_locstar_and_bounded}, 
$\contactpoint{\loglego{f}}{\loglego{g}}$ is a \wlocstar$\loglego{f}$,  
and $\contactset{\loglego{(f^q)}}{\loglego{(g^q)}}$ is bounded by convexity. 
Thus, there are  contact pairs 
			$(\tilde{u}_1,\tilde{v}_1)$, $\dots$,  
			$(\tilde{u}_m,\tilde{v}_m)$ $\in \contactset{f^q}{g^q}$
			and positive weights $\tilde{c}_1,\dots,\tilde{c}_m$  satisfying
		\eqref{eq:functional-glmp-lowner_concave}. 
		The necessity of condition \eqref{eq:functional_glmp-lowner-conc_intro} follows from \Href{Lemma}{lem:equiv_lowner_cond_power}.

By the argument from \Href{Section}{sec:radially_symmetric} for radially symmetric,
 function the sufficiency of the corresponding conditions follows as well.
\end{proof}

\section{Discussion}\label{sec:corollaries_and_discussion}

\subsection{Convex sets as a special case}

Observe that \Href{Theorem}{thm:GLMP} immediately follows from
\Href{Theorem}{thm:john_condition_general} by setting $f = \chi_{L}$ and $\gbase 
= \chi_{\operatorname{ext}{K}}$ therein.

\subsection{Equivalence of the John and the L\"owner problems}\label{sec:johnlownerequivalence}
Let $f,g\colon\Red\to[0,+\infty)$ satisfy the assumptions of \Href{Theorem}{thm:lowner_condition_general}. 
Additionally, let $f$ satisfy Basic assumptions \eqref{assumptions:basic}, $\contactpoint{f}{g}$ be a \wlocstar$g$, and  $\contactset{f}{g}$ be bounded.
Clearly, if $g$ is a local minimizer in the L\"owner $s$-problem \eqref{eq:lowner_problem_intro}, then $f$ is a local maximizer in the John $s$-problem \eqref{eq:john_problem_intro} for $g$ and $f$. Hence, both  \Href{Theorem}{thm:lowner_condition_general} and  \Href{Theorem}{thm:john_condition_general} are applicable in this case. We show that the corresponding versions of 
\eqref{eq:john_problem_intro} and \eqref{eq:lowner_problem_intro} are equivalent in this case.

Since $\supp{\loglego{g}}$ is bounded, we have that $\supp g = \supp \loglego{f} = \R^d$. Therefore,  there are no ``flat'' contact points neither for the pair $f$ and $g,$ nor for $\loglego{f}$ and $\loglego{g}$. More precisely,
if  $\parenth{\upthing{u}, \upthing{v}}$ with $\upthing{u} = (u, f(u))$  and $\upthing{v} = (v, \nu)$ belongs to  $\contactset{f}{g}$ (resp. $\contactset{\loglego{g}}{\loglego{f}}$)
then $\nu > 0$ and $f(u) > 0$  (resp. $\loglego{f}(u) > 0$).
In this case,  \Href{Lemma}{lem:v_via_grad} yields 
 \[
 \contactopjohn{\upthing{u}}{\upthing{v}} =  
\frac{\parenth{({u} \otimes {p}) \oplus 1, p}}{1 + \iprod{p}{u}} 
 \]
for some  $p \in \partial{\parenth{- \ln f}}(u) \subset \partial{\parenth{- \ln g}}(u).$ 
By the properties of the subdifferential listed in \Href{Lemma}{lem:subdif_basics}, 
\[
 u \in \partial{\parenth{- \ln \loglego{g}}}(p) \subset \partial{\parenth{- \ln \loglego{f}}}(p)\] 
 and
$ \loglego{g}(p) =  \loglego{f}(p) =  \frac{e^{- \iprod{p}{u}}}{f(u)}.$ 
Hence, the pair $\parenth{\upthing{u}, \upthing{v}}$ belongs to $\contactset{f}{g}$ 
if and only if the pair 
$\parenth{\tilde{u}, \tilde{v}}$ belongs to $\contactset{\loglego{g}}{\loglego{f}},$
where 
$\tilde{u} = (p, \loglego{g}(p))$  and 
$\tilde{v} =  \frac{\parenth{u, \frac{1}{\loglego{g}(p)} }}{1 + \iprod{p}{u}}.$
By direct calculations,
 \[
 \contactopjohn{\upthing{u}}{\upthing{v}} =  
 \contactoplowner{\tilde{u}}{\tilde{v}}. 
 \]
The desired equivalence follows.

\subsection{The fixed center John and L\"owner problems -- no translation}
Our extended contact operator $(A \oplus \alpha, a)\in\MM$ consists of two parts: 
the operator part $A \oplus \alpha$ and the translation part $a$.
In essence, only the rightmost equation in \eqref{eq:john_problem_intro} (resp.  \eqref{eq:lowner_problem_intro}) is responsible for the translation (shifting) of the function. If $A$ is a non-singular matrix of order $d$ and $\alpha > 0$,
then $\lifting{\alpha g(Ax)}$ is the linear image of $\lifting{g}$.

We will say that we consider the John or the L\"owner $s$-problem (resp., 
Positive John/L\"owner $s$-problem)  
\emph{with fixed center} if we maximize or minimize over $\funposfc{g}$ (resp., $\funpposfc{g}$), where
$\funposfc{g}=\{\alpha g(Ax+a)\st (A\oplus\alpha, a)\in\MM_{f.c.}\}$
and
$\funpposfc{g}=\{\alpha g(Ax+a)\st (A\oplus\alpha, a)\in\MM_{f.c.}^+\}$,
where $\MM_{f.c.} = \left\{(A \oplus \alpha, 0) \st (A \oplus \alpha, 0) \in \MM \right\}$  and 
$\MM_{f.c.}^{+} = \left\{(A \oplus \alpha, 0) \st (A \oplus \alpha, 0) \in \MM^{+} \right\}$.

\begin{thm}[Fixed center John condition]\label{thm:john-condition-fixed_center}
Let the functions $f,\gbase\colon\Red\to[0,+\infty)$ satisfy the assumptions of \Href{Theorem}{thm:john_condition_general}.
Set $g =\logconc{\gbase}$. Then the following hold.
	\begin{enumerate}
		\item
			If $h=g$  is  a local maximizer in the John $s$-problem with fixed center,
				then there exist contact pairs 
			$(\upthing{u}_1,\upthing{v}_1), \dots,  
			(\upthing{u}_m,\upthing{v}_m) \in \contactset{f}{\gbase}$
			and positive weights $c_1,\ldots,c_m$ such that
		\begin{equation}\label{eq:functional_glmp_fixed_center}
		\sum_{i=1}^{m} c_i {u}_i \otimes {v}_i = 
		 \id_{d} 				
		 	\quad\text{and}\quad 
		 \sum_{i=1}^{m} c_i f(u_i)\nu_i =  s
				\end{equation}
        where $\upthing{u}_i=(u_i, f(u_i))$ and $\upthing{v}_i=(v_i,\nu_i)$.
		\item
			If there exist contact pairs and positive weights  satisfying  equation \eqref{eq:functional_glmp_fixed_center},
			then $h=g$ is a global maximizer in the Positive John $s$-problem with fixed center. 
	\end{enumerate}
\end{thm}
\begin{proof}[Sketch of the proof]
In \Href{Lemma}{lem:separation_John_problem}, 
we consider the linear span of $\MM_{f.c.},$ which is a subspace of $\MMbar.$
By a standard separation argument, 
 there are no contact pairs of $f$ and $\gbase$ and positive weights satisfying equation \eqref{eq:functional_glmp_fixed_center} if and only if, a linear hyperplane 
 with normal of the form $(H \oplus \gamma, 0)$ strongly separates the set of contact operators and $(\id \oplus s, 0).$  
 The rest of the proof coincides with that of \Href{Theorem}{thm:john_condition_general}.
\end{proof}

Similarly, we have the following.

\begin{thm}[Fixed center L\"owner condition]\label{thm:lowner_condition_fixed_center}
Let the functions $f,\gbase\colon\Red\to[0,+\infty)$ satisfy the assumptions of \Href{Theorem}{thm:lowner_condition_general}.
Then the following hold.
	\begin{enumerate}
		\item\label{item:lowner_fixed_center}
			If $h=g$ is a local minimizer in the L\"owner $s$-problem with fixed center,
			then there exist contact pairs 
			$(\upthing{u}_1,\upthing{v}_1), \dots,  
			(\upthing{u}_m,\upthing{v}_m) \in \contactset{\loglego{g}}{\loglego{f}}$
			and positive weights $c_1,\ldots,c_m$ such that
		\begin{equation}\label{eq:functional_glmp_lowner_fixed_c}
		\sum_{i=1}^{m} c_i {v}_i \otimes {u}_i = 
		 \id_{d}, \quad 
		 \sum_{i=1}^{m} c_i \loglego{g}(u_i) \cdot \nu_i =  s.
        \end{equation} 
        where $\upthing{u}_i=(u_i, \loglego{g}(u_i))$ and $\upthing{v}_i=(v_i,\nu_i)$.
		\item
			If there exist contact pairs and positive weights satisfying equation \eqref{eq:functional_glmp_lowner_fixed_c},
			then $h=g$ is a global maximizer in the Positive position L\"owner $s$-problem with fixed center.
	\end{enumerate}
\end{thm}

We note that in the case of fixed center problems, the corresponding conditions, 
that is, equations \eqref{eq:functional_glmp_fixed_center} and 
\eqref{eq:functional_glmp_lowner_fixed_c} coincide whenever both theorems are 
applicable.

\subsection*{Acknowledgement} We thank Alexander Litvak for the many discussions on Theorem~\ref{thm:GLMP}. Igor Tsiutsiurupa participated in the early stage of this project.
To our deep regret, Igor chose another road for his life and stopped working with us.

\bibliographystyle{amsalpha} 
\bibliography{uvolit}

\end{document}